\documentclass[a4paper,11pt]{article}

\usepackage[a4paper,left=26.5mm,right=26mm,top=26.25mm,bottom=36.75mm]{geometry}

\usepackage[english]{babel}	
\usepackage{lmodern}
\usepackage[T1]{fontenc}
\usepackage[utf8]{inputenc}	

\usepackage{lipsum}
\usepackage{amsfonts}
\usepackage{algorithmic}

\usepackage{authblk}
\usepackage{hyperref}

\title{Luenberger observers for infinite-dimensional systems, Back and Forth Nudging and application to a crystallization process\thanks{This research was partially funded by the French Grant ANR ODISSE (ANR-19-CE48-0004-01)}}

\author[$*$1]{Lucas Brivadis}
\author[1]{Vincent Andrieu}
\author[1]{Ulysse Serres}
\author[2]{Jean-Paul Gauthier}

\affil[1]{Univ. Lyon, Universit\'e Claude Bernard Lyon 1, CNRS, LAGEPP UMR 5007, France}
\affil[2]{Universit\'e de Toulon, Aix Marseille Univ, CNRS, LIS, France}
\affil[$*$]{lucas.brivadis@univ-lyon1.fr}

\usepackage{amsopn}

\usepackage{amsfonts, amssymb, amsmath, amsthm}
\allowdisplaybreaks
\usepackage[all]{xy} 
\usepackage{mathrsfs}		
\usepackage{bbm}			
\DeclareMathAlphabet{\mathbbold}{U}{bbold}{m}{n}	
\usepackage{verbatim}
\usepackage{bm}	
\usepackage{comment}
\usepackage{accents} 

\usepackage{graphics}              
\usepackage{import}              
\usepackage{xcolor}            
\usepackage{graphicx}            

\usepackage{xspace} 

\usepackage{centernot} 

\usepackage{enumitem}

\usepackage[nocompress]{cite}

\usepackage{xfrac} 

\catcode`\@=11
\def\downparenfill{$\m@th\braceld\leaders\vrule\hfill\bracerd$}
\def\overparen#1{\mathop{\vbox{\ialign{##\crcr\crcr
\noalign{\kern0.4ex}
\downparenfill\crcr\noalign{\kern0.4ex\nointerlineskip}
$\hfil\displaystyle{#1}\hfil$\crcr}}}\limits}
\catcode`\@=12

\usepackage{thm-restate} 

\newcommand{\set}[1]{\left\{#1\right\}}
\newcommand{\setst}[2]{\left\{#1~\middle\vert~#2\right\}}

\newcommand{\I}{\chi}

\newcommand{\lc}{[}

\newcommand{\R}{\mathbb{R}}
\newcommand{\N}{\mathbb{N}}

\renewcommand{\le}{\leqslant}
\renewcommand{\ge}{\geqslant}
\renewcommand{\leq}{\leqslant}
\renewcommand{\geq}{\geqslant}

\newcommand{\eps}{\varepsilon}
\renewcommand{\epsilon}{\varepsilon}
\renewcommand{\phi}{\varphi}

\newcommand{\lang}{\left\langle}
\newcommand{\rang}{\right\rangle}

\newcommand{\ps}[2]{\lang#1,#2\rang}
\newcommand{\psX}[2]{\lang#1,#2\rang_\XX}
\newcommand{\psY}[2]{\lang#1,#2\rang_\YY}

\newcommand{\norm}[1]{\left\Vert #1\right\Vert_\XX}
\newcommand{\normY}[1]{\left\Vert #1\right\Vert_\YY}

\newcommand{\abs}[1]{\left\vert #1\right\vert}

\newcommand{\dd}{\mathrm{d}}
\newcommand{\diff}{\dd}

\newcommand{\cv}{\to}
\newcommand{\cvl}[1]{\underset{#1}{\longrightarrow}}
\newcommand{\cvf}{\overset{w}{\rightharpoonup}}
\newcommand{\cvfl}[1]{\underset{#1}{\overset{w}{\relbar\joinrel\relbar\joinrel\rightharpoonup}}}

\newcommand{\fonction}[5]{
\begin{align*}
\displaystyle
\begin{array}{lrcl}
#1: & #2 & \longrightarrow & #3 \\
    & #4 & \longmapsto & #5
\end{array}
\end{align*}}

\newcommand{\XX}{X}
\newcommand{\YY}{Y}

\newcommand{\ie}{\textit{i.e.,~}}
\newcommand{\eg}{\textit{e.g.,~}}

\newcommand{\st}{such that~}

\newcommand{\Id}{\mathrm{Id}} 

\newcommand{\lin}{\mathscr{L}}
\newcommand{\linx}{\lin(\XX)}
\newcommand{\linxy}{\lin(\XX, \YY)}

\newcommand{\opa}{A}
\newcommand{\opc}{C}
\newcommand{\dom}{\mathcal{D}}
\newcommand{\doma}{\dom}
\newcommand{\gram}{W}

\newcommand{\sg}{\mathbb{T}}
\newcommand{\sgeps}{\mathbb{S}}
\newcommand{\sgepsi}{\mathbb{S}_+}
\newcommand{\sgepsii}{\mathbb{S}_-}

\newcommand{\evol}{\mathbb{T}}
\newcommand{\evoleps}{\mathbb{S}}
\newcommand{\evolepsi}{\mathbb{S}_+}
\newcommand{\evolepsii}{\mathbb{S}_-}

\newcommand{\regis}{\textsuperscript{\mbox{\scriptsize{\textregistered}}}\xspace}
\newcommand{\fbrm}{FBRM\regis}
\newcommand{\noy}{k}
\newcommand{\dlc}{q}

\newcommand{\dtc}{n}
\newcommand{\vit}{G}
\newcommand{\cont}{u}

\newcommand{\etat}{z}
\newcommand{\mes}{y}
\newcommand{\rr}{r}
\newcommand{\etath}{\hat{\etat}}
\newcommand{\xmin}{x_{\min}}
\newcommand{\xmax}{x_{\max}}
\newcommand{\lmin}{\ell_{\min}}
\newcommand{\lmax}{\ell_{\max}}

\newcommand{\xd}{x_0}
\newcommand{\xf}{x_1}

\newcommand{\interv}{U}
\newcommand{\intervmax}{\interv_{\max}}

\newcommand{\ff}{f}
\newcommand{\f}{u}

\newcommand{\ceps}{\opc}
\newcommand{\trans}{v}
\newcommand{\til}{\tilde}

\newcommand{\Obsspace}{\mathcal{O}}
\newcommand{\opl}{L}
\newcommand{\proj}{\Pi}
\newcommand{\pobs}{\proj_\Obsspace}
\newcommand{\pobst}{\proj_{\Obsspace_\tau}}

\newcommand{\bornL}{M_L}

\DeclareMathOperator*{\esssup}{ess\:sup}

\DeclareMathOperator{\ran}{ran}

\DeclareMathOperator{\modulo}{\,mod\,}

\DeclareMathAlphabet{\mathbbold}{U}{bbold}{m}{n}	

\theoremstyle{plain}
\newtheorem{theorem}{Theorem}[section]
\newtheorem{corollary}[theorem]{Corollary}
\newtheorem{lemma}[theorem]{Lemma}
\newtheorem{proposition}[theorem]{Proposition}

\theoremstyle{definition}
\newtheorem{definition}[theorem]{Definition}

\newtheorem{remark}[theorem]{Remark}

\numberwithin{equation}{section} 
\begin{document}

\maketitle

\begin{abstract}
This paper deals with the observer design problem for time-varying linear infinite-dimensional systems.
We address both the problem of online estimation of the state of the system from the output via an asymptotic observer,
and the problem of offline estimation of the initial state via a Back and Forth Nudging (BFN) algorithm.
In both contexts, we show under a weak detectability assumption that a Luenberger-like observer may reconstruct the so-called observable subspace of the system.
However, since no exact observability hypothesis is required, only a weak convergence of the observer holds in general. Additional conditions on the system are required to show the strong convergence.
We provide an application of our results to a batch crystallization process modeled by a one-dimensional transport equation with periodic boundary conditions, in which we try to estimate the Crystal Size Distribution from the Chord Length Distribution.
\end{abstract}

\section{Introduction}

To analyze, monitor or control physical or biological phenomena, the first step is to provide a mathematical modeling  in the form of mathematical equations that describe the evolution of the system variables. 
Some of these variables are accessible through measurement and others are not.
One of the problems in control engineering is that of designing algorithms to provide real time estimates of the unmeasured data from the others.
These estimation algorithms are called state observers and can be found in many devices.
The implementation of such observers for infinite-dimensional systems is a topic of great interest from both the practical and theoretical points of view that has been extensively studied in the past decades (see, \eg \cite{TW, slemrod1974note, liu1997locally, slemrod1972linear, Celle, xu1995observer}).

More recently, these results have been employed in data assimilation problems, leading to the so-called Back and Forth Nudging (BFN) algorithms (see, \eg \cite{auroux2005back, Ramdani, haine2014recovering, haine2011, haine2011fr, ito2011time, Boulanger}).
In this context, observers are used iteratively forward and backward in time to solve the offline estimation problem of reconstructing the initial state of the system. Such problems occur for example in meteorology or oceanography \cite{auroux2008nudging, auroux2009back}.

Mainly two types of results are known about the convergence of Luenberger-like observers for linear systems, depending on the observability hypotheses made.
Under an \emph{exact observability} hypothesis that links the $L^2$-norm of the measured output on some time interval to the norm of the initial state, a Luenberger-like asymptotic observer that converges \emph{exponentially} to the actual state of the system may be designed \cite{slemrod1972linear, TW, liu1997locally}.
Under this hypothesis, it is proved in \cite{Ramdani, ito2011time} that the BFN algorithm estimates exponentially the initial state of the system.

Otherwise, when the system is only \emph{approximately observable}, that is, any two trajectories of the system may be distinguished by looking at the output on some time interval, then, if the system is \emph{dissipative}, one can prove that 
the same asymptotic observer converges only \emph{weakly} to the state \cite{slemrod1974note, xu1995observer, Celle}.
For the BFN algorithm, G. Haine proved in \cite{haine2014recovering} for autonomous systems generated by skew-adjoint generators that the initial state estimation still converges \emph{strongly} (but no more exponentially) to the actual initial state. 

The time-varying context has been investigated for control systems in \cite{xu1995observer, Celle}, in which some \emph{persistency} assumptions are required, and a weak convergence is guaranteed.
When no observability assumptions are made, then one may expect the observer to converge to the so-called \emph{observable subspace} of the system, which is clearly defined only for autonomous systems.

In this paper, we consider infinite-dimensional time-varying linear systems.
We investigate both the usual asymptotic observer design problem, and the backward and forward observers design problem for the BFN algorithm.
We relax the dissipativity hypothesis, and replace it by a weak detectability hypothesis, which states that the distance between any two trajectories of the system that share the same output is a non-increasing function of time.
When no observability hypothesis holds, we show that the observer estimates in the weak topology the observable part of the state, which is equal to the whole state when the system is approximately observable. Under additional assumptions on the system, we also show the strong convergence of the observer.
We compare our results with the existing literature mentioned above.

As an application of our results, we consider a batch crystallization process modeled by a one-dimensional time-varying transport equation with periodic boundary conditions.
This process aims to produce solid crystals meeting some physical and chemical specifications.
One of the most important physical property to monitor is the Crystal Size Distribution (CSD).
Information available online are the Chord Length Distribution (CLD) obtained from the \fbrm technology and the solute concentration. However, as shown in the following, the considered model describing this system is time-varying and not exactly observable, which is a motivation for these theoretical developments.

The paper is organized as follows.
In Section~\ref{secprob}, we describe the systems under consideration, and make the required assumptions to ensure the well-posedness of the usual asymptotic observer and the backward and forward observers of the BFN.
Our main results are stated in Section~\ref{secmain}, discussed in Section~\ref{secdisc}, and proved in Section~\ref{secproof}.
In Section~\ref{secex}, we discuss about their implications for the one-dimensional transport equation with periodic boundary conditions and to a batch crystallization process, in which we aim to estimate the CSD from the CLD.

\medskip
\paragraph{Notations}
Denote by $\R$ (resp. $\R_+$) the set of real (resp. non-negative) numbers and by $\N$ (resp. $\N^*$) the set of non-negative (resp. positive) integers.
For all Hilbert space $X$, denote by $\psX{\cdot}{\cdot}$ the inner product over $\XX$ and $\norm{\cdot}$ the induced norm.
The identity operator over $\XX$ is denoted by $\Id_\XX$.
For all $k\in\N\cup\{\infty\}$ and all interval $U\subset\R$, the set $C^k(U; \XX)$ is the set of $k$-continuously differentiable functions from $U$ to $\XX$.

We recall the characterization of the strong and weak topologies on $\XX$.
A sequence $(x_n)_{n\geq0}\in\XX^\N$ is said to be strongly convergent to some $x^\star\in\XX$ if
$\norm{x_n-x^\star}\cv 0$ as $n\cv+\infty$, and we shall write $x_n\cv x^\star$ as $n\cv+\infty$.
It is said to be weakly convergent to $x^\star$ if
$\psX{x_n-x^\star}{\psi}\cv 0$ as $n\cv+\infty$ for all $\psi\in\XX$, and we shall write $x_n\cvf x^\star$ as $n\cv+\infty$.
The strong topology on $\XX$ is finer than the weak topology (see, \eg \cite{Brezis} for more properties on these usual topologies).

If $\YY$ is also a Hilbert space, then $\linxy$ denotes the space of \emph{bounded} linear maps from $\XX$ to $\YY$ and $\|\cdot\|_{\linxy}$ the operator norm. Set $\linx = \lin(\XX, \XX)$.
For all $\opl\in\linxy$, denote by $\ran \opl$ its range and $\ker \opl$ its kernel.
We identify the Hilbert spaces with their dual spaces via the canonical isometry, so that the adjoint of $\opl$, denoted by $\opl^*$, lies in $\lin(\YY, \XX)$.
If $\opl^*\opl = \opl\opl^*$, then $\opl$ is said to be \emph{normal}.
If there exists a positive constant $\alpha$ such that $\norm{\opl x}\geq\alpha\norm{x}$ for all $x\in\XX$, then $\opl$ is said to be \emph{bounded from below}.

For any set $E\subset\XX$, the closure of $E$ in the strong topology of $\XX$ is denoted by $\overline{E}$.
If $E$ is a linear subspace of $\XX$, then $E^\perp$ denotes its orthogonal complement in $\XX$.
Moreover, if $E$ is closed, set $\proj_E\in\linx$ the orthogonal projection such that $\ran \proj_E = E$.

\section{Problem statement}\label{secprob}

Let $\XX$ and $\YY$ be two Hilbert spaces with real\footnote{Even if we could consider complex inner product, we prefer to restrict ourselves to real inner products to simplify the presentation.} inner products.
Let $\dom$ be a dense subset of $\XX$.
For all $t\geq0$, let $\opa(t): \doma \to \XX$ be the generator of a strongly continuous semigroup on $\XX$ and $\opc\in \linxy$.
Let $z_0\in\XX$.
Consider the non-autonomous linear abstract Cauchy problem with measured output
\begin{equation}
\begin{aligned}
\begin{cases}
\dot \etat = \opa(t) \etat\\
\etat(0) = \etat_0
\end{cases}
\end{aligned}
,\qquad
y = Cz.
\label{syst}
\end{equation}
In this paper, we are concerned with the problem of designing an observer of the state $\etat$ based on the measurement $\mes$. 
We make the following hypotheses, which make \eqref{syst} a case of hyperbolic system as defined in \cite[Chapter 5]{Pazy}.
Let $T\in\R_+\cup\{+\infty\}$, and adopt the convention that $[0, T] = \R_+$ if $T=+\infty$.
Assume that the family $(\opa(t))_{t\in[0, T]}$ is a stable (see \cite[Chapter 5, Section 5.2]{Pazy} for a definition) family of generators of strongly continuous semigroups that share the same domain $\dom$. Assume also that for all $x\in\doma$, the function $t\mapsto\opa(t)x$ is continuously differentiable on $\XX$. These hypotheses hold for the rest of the paper.
Then \cite[Chapter 5, Theorem 4.8]{Pazy} ensures that the family $(\opa(t))_{t\in[0, T]}$ is the generator of a unique evolution system on $\XX$ denoted by $(\evol(t, s))_{0\leq s\leq t\leq T}$.
Moreover,
there exist two constants $M,\:\omega>0$ such that
\begin{equation}\label{E:bound-M-omega}
    \|\evol(t, s)\|_{\linx} \leq M e^{\omega(t-s)}, \quad \forall~0\leq s\leq t\leq T.
\end{equation}
For all $\etat_0\in\XX$,~\eqref{syst} admits a unique solution $\etat\in C^0([0, T]; \XX)$ given by $\etat(t) = \evol(t, 0)\etat_0$ for all $t \in[0, T]$. Moreover, if $\etat_0\in\doma$, then $\etat\in C^0([0, T]; \doma)\cap C^1([0, T]; \XX)$.
The reader may refer to \cite[Chapter 5]{Pazy} or \cite{ito} for more details on the evolution equations theory.

\begin{definition}[Autonomous context]\label{rem:auto}
We shall say that~\eqref{syst} is autonomous if there exists an operator $A:\doma\rightarrow\XX$ such that $A(t)=A$ for all $t\in\R_+$.
\end{definition}
\begin{remark}
In the autonomous context, $T=+\infty$ and the evolution system $\evol$ is such that
$\evol(t, s) = \evol(t-s, 0)$
for all $t\geq s\geq0$.
By abuse of notation, the strongly continuous semigroup generated by $A$ is also denoted by $\evol$, so that $\evol(t) = \evol(t, 0)$ for all $t\in\R_+$.
The same shortened notations hold for any other autonomous system.
\end{remark}

Our goal is to build an observer system $\etath$ fed by the output $\mes$ of~\eqref{syst}, such that $\etath$ estimates the actual state $\etat$.
We raise two different observer issues: the usual asymptotic observer problem, and the inverse problem of reconstructing the initial state.

\subsection{Asymptotic observer}
In order to find an asymptotic observer, we naturally assume that $T = +\infty$.
The goal is to find a new dynamical system fed by the output of~\eqref{syst} which asymptotically learns the state from the dynamic of the output.
This issue was raised by D. Luenberger in his seminal paper \cite{Luenberger} in the context of finite-dimensional autonomous linear systems. In \cite{slemrod1972linear, slemrod1974note}, J. Slemrod investigates the dual problem of stabilization in infinite-dimensional Hilbert spaces. In this paper, we follow this path and introduce the usual infinite-dimensional version of the Luenberger observer.

Let $\rr>0$ and $\etath_0\in\XX$. Consider the following Luenberger-like observer
\begin{equation}
\begin{aligned}
\begin{cases}
\dot \etath = \opa(t) \etath - \rr\opc^*(\opc\etath-\mes)\\
\etath(0) = \etath_0
\end{cases}
\end{aligned}
\label{obs}
\end{equation}
Set $\eps = \etath-\etat$ and $\eps_0 = \etath_0-\etat_0$.
From now on, $\etath$ represents the state estimation made by the observer system and $\eps$ the error between this estimation and the actual state of the system.
Then $\etath$ satisfies~\eqref{obs} if and only if $\eps$ satisfies
\begin{equation}
\begin{aligned}
\begin{cases}
\dot \eps = (\opa(t) - \rr\opc^*\opc)\eps\\
\eps(0) = \eps_0
\end{cases}
\end{aligned}
\label{eps}
\end{equation}
Since $\opc\in\linxy$, \cite[Chapter 5, Theorem 2.3]{Pazy} claims that $(\opa(t)-\rr\opc^*\opc)_{t\ge 0}$ is also a stable family of generators of strongly continuous semigroups, and generates an evolution system on $\XX$ denoted by $(\evoleps(t, s))_{0\leq s\leq t}$.
Then,
systems \eqref{obs} and~\eqref{eps} have respectively a unique solution $\etath$ and $\eps$ in $C^0([0, +\infty); \XX)$.
Moreover,
$\etath(t) =\evol(t, 0)\etat_0+\evoleps(t, 0)\eps_0$
and $\eps(t) = \evoleps(t, 0)\eps_0$ for all $t \in[0, +\infty)$. If $(\etath_0, \eps_0)\in\doma^2$, then $\etath$, $\eps\in C^0([0, +\infty); \doma)\cap C^1([0, +\infty); \XX)$.

We are interested in the convergence properties of the state estimation $\etath$ to the actual state $\etat$, \ie of the estimation error $\eps$ to $0$.

\begin{definition}[Asymptotic observer]
For any closed linear subspace $\Obsspace$ of $\XX$,
\eqref{obs} is said to be a strong (resp. weak) asymptotic $\Obsspace$-observer of~\eqref{syst} if and only if $\pobs\sgeps(t, 0)\eps_0\cv0$ (resp. $\pobs\sgeps(t, 0)\eps_0\cvf0$) as $t\cv+\infty$ for all $\eps_0\in\XX$.
An $\XX$-observer is shortly called an observer.
\end{definition}

\subsection{Back and forth nudging}
Now consider a problem which is slightly different from the former one.
Assume that $T<+\infty$, and address the problem of offline state estimation.
The goal is to use the knowledge of the output and its dynamic on the finite time interval $[0, T]$ to estimate the initial state of the system.
To achieve this, the idea is to use iteratively forward and backward observers.
This methodology is called the back and forth nudging in \cite{auroux2005back, auroux2012back, auroux2008nudging}, or the time reversal based algorithm in \cite{ito2011time}.

In order to build this observer, we need to assume that the family $(\opa(t))_{t\in[0, T]}$ is the generator of a \emph{bi-directional} evolution system on $\XX$ denoted by $(\evol(t, s))_{0\leq s, t\leq T}$.
\begin{definition}[Bi-directional evolution system]\label{def:bi}
A bi-directional evolution system on a Banach space $X$ over $[0, T]$ is a two parameter family $(\evol(t, s))_{0\leq s, t\leq T}$ of operators in $\linx$ satisfying the three following properties:
\begin{enumerate}[label = (\roman*)]
    \item for all $t\in[0, T]$, $\sg(t, t) = \Id_\XX$;
    \item for all $t, s, \tau\in[0, T]$, $\sg(t, s)\sg(s, \tau) = \sg(t, \tau)$;
    \item for all $t\in[0, T]$ and all $x\in X$, $\sg(t, s)x \to x$ as $s\to t$.
\end{enumerate}
\end{definition}
Using the classical definition of an evolution system \cite[Definition 5.3]{Pazy}, one can check that 
a family $(\evol(t, s))_{0\leq s, t\leq T}$
of bounded linear operators on $X$ is a bi-directional evolution system
if and only if:
\begin{enumerate}[label = (\alph*)]
    \item $(\evol(t, s))_{0\leq s\leq t\leq T}$ is an evolution systems on $\XX$;
    \item $(\evol(T-t, T-s))_{0\leq s\leq t\leq T}$ is an evolution system on $\XX$;\label{enum:b}
    \item for all $t, s\in [0, T]$, $\sg(s, t)\sg(t, s) = \Id_\XX$.\label{enum:c}
\end{enumerate}
A family of operators $(A(t))_{t\in[0, T]}$ is said to be the generator of a bi-directional evolution system if and only if it is the generator of an evolution system $(\evol(t, s))_{0\leq s\leq t\leq T}$, $(-\opa(T-t))_{t\in[0, T]}$ is the generator of an evolution system $(\evol(T-t, T-s))_{0\leq s\leq t\leq T}$ and \ref{enum:c} is satisfied.
Each time backward and forward observers are considered, we assume that $(\opa(t))_{t\in[0, T]}$ is the generator of a bi-directional evolution system on $\XX$.
Moreover, we make on the family $(-\opa(T-t))_{t\in[0, T]}$ the same hypothesis than on $(A(t))_{t\in[0, T]}$, namely, it is a stable family of generators of strongly continuous semigroups that share the same domain $\dom$, so that \eqref{E:bound-M-omega} is now satisfied for all $s, t\in[0, T]$.

Let $\etath_0\in\XX$. For every $n\in\N$, we consider the following dynamical systems defined on $[0, T]$ as in \cite{Ramdani} by
\begin{align}
&\begin{cases}
\dot {\etath}^{2n} = \opa(t) \etath^{2n} - r\opc^*(\opc\etath^{2n}-\mes)\\
\etath^{2n}(0) =
\begin{cases}
\etath^{2n-1}(0) & \text{if } n\geq 1\\
\etath_0 & \text{otherwise.}
\end{cases}
\end{cases}
\label{obs2}\\
&\begin{cases}
\dot {\etath}^{2n+1} = \opa(t) \etath^{2n+1} + r\opc^*(\opc\etath^{2n+1}-\mes)\\
\etath^{2n+1}(T) = \etath^{2n}(T).
\end{cases}
\label{obs2b}
\end{align}
For all $n\in\N$, let $\epsilon^n = \etath^n - \etat$ and $\epsilon_0 = \etath_0-\etat_0$. Then $\etath^{2n}$ and $\etath^{2n+1}$ satisfy respectively~\eqref{obs2} and~\eqref{obs2b} if and only if $\epsilon^{2n}$ and $\epsilon^{2n+1}$ satisfy
\begin{align}
&\begin{cases}
\dot {\epsilon}^{2n} = (\opa(t) - r\opc^*\opc) {\epsilon}^{2n}\\
{\epsilon}^{2n}(0) =
\begin{cases}
{\epsilon}^{2n-1}(0) & \text{if } n\geq 1\\
{\epsilon}_0 & \text{otherwise.}
\end{cases}
\end{cases}
\label{eps2}\\
&\begin{cases}
\dot {\epsilon}^{2n+1} = (\opa(t) + r\opc^*\opc) {\epsilon}^{2n+1}\\
{\epsilon}^{2n+1}(T) = {\epsilon}^{2n}(T).
\end{cases}
\label{eps2b}
\end{align}

Since $\opc\in\linxy$, \cite[Chapter 5, Theorem 2.3]{Pazy} claims that both $(\opa(t)-\rr\opc^*\opc)_{t\in[0, T]}$ and $(\opa(t)+\rr\opc^*\opc)_{t\in[0, T]}$ are stable families of generators of strongly continuous semigroups that generate bi-directional evolution systems on $\XX$ denoted respectively by $(\evolepsi(t, s))_{0\leq s, t\leq T}$ and $(\evolepsii(t, s))_{0\leq s, t\leq T}$.
Then, for all $n\in\N$, \eqref{obs2}, \eqref{obs2b}, \eqref{eps2} and \eqref{eps2b} have respectively a unique solution ${\etath}^{2n}$, ${\etath}^{2n+1}$, ${\epsilon}^{2n}$ and ${\epsilon}^{2n+1}$ in $C^0([0, T]; \XX)$.
Moreover,
$\etath^{2n}(t) =
\sg(t, 0)\etat_0 +
\sgepsi(t, 0)\eps^{2n}(0)$,
$\etath^{2n+1}(t) = 
\sg(t, T)\etat(T) +
\sgepsii(t, T)\eps^{2n+1}(T)$,
$\epsilon^{2n}(t) = \sgepsi(t, 0)\epsilon^{2n}(0)$, $\epsilon^{2n+1}(t) = \sgepsii(t, T)\epsilon^{2n+1}(T)$
for all $t\in[0, T]$.
In particular, note that
\begin{equation}
    \eps^{2n}(0) = \left(\sgepsii(0, T)\sgepsi(T, 0)\right)^{n}\eps_0.
\end{equation}
If $(\etath_0, \eps_0)\in\doma^2$, then ${\etath}^{n}, {\epsilon}^{n} \in C^0([0, T]; \doma)\cap C^1([0, T]; \XX)$ for all $n\in\N$.

We are interested in the convergence properties of the initial state estimation $\etath^{2n}(0)$ to the actual state $\etat(0)$, \ie of the estimation error $\eps^{2n}(0)$ to $0$, as $n$ goes to infinity.

\begin{definition}[Back and forth observer]
For any closed linear subspace $\Obsspace$ of $\XX$,
the system \eqref{obs2}-\eqref{obs2b}
is said to be a strong (resp. weak) back and forth $\Obsspace$-observer of~\eqref{syst} if and only if $\pobs\eps^{2n}(0)\cv0$ (resp. $\pobs\eps^{2n}(0)\cvf0$) as $n\cv+\infty$ for all $\eps_0\in\XX$.
An $\XX$-observer is shortly called an observer.
\end{definition}

\section{Main results}\label{secmain}

In this section, we state our main results about the asymptotic observer and the back and forth observer.
Then, we discuss our hypotheses and compare our results with the existing literature.

A crucial operator to consider in order to investigate the convergence properties of a Luenberger-like observer is the so-called \emph{observability Gramian}.
\begin{definition}[Observability Gramian]
For all $t_0\in[0, T]$ and all $\tau\in[0, T-t_0]$, let us define
\fonction{\gram(t_0, \tau)}{\XX}{\XX}{\etat_0}{\displaystyle\int_{t_0}^{t_0+\tau}\evol(t, t_0)^*\opc^*\opc\evol(t, t_0)\etat_0\dd t}
the \emph{observability Gramian} of the pair $(\sg, \opc)$.
\end{definition}
The operator $\gram(t_0, \tau)$ is a bounded self-adjoint endomorphism of $\XX$, that characterizes the observability properties of~\eqref{syst}.
Moreover, $\gram$ is continuous in $\linx$ with respect to $(t_0, t)$, and we have 
$\|\gram(t_0, \tau)\|_{\linx} \leq \left(Me^{\omega\tau}\|C\|_{\linxy}\right)^2$.

\begin{remark}
In the autonomous context,
$\gram(t_0, \tau) = \gram(0, \tau)$ for all $t_0,\tau\in\R_+$.
Then, by abuse of notation, we denote $\gram(\tau) = \gram(0, \tau)$.
\end{remark}

\begin{definition}[Observable subspace]\label{defO}
 For all $\tau\in[0, T]$, let
\begin{equation}\label{defobssapce}
    \Obsspace_\tau = \left(\ker\gram(0, \tau)\right)^\perp.
\end{equation}
be the \emph{observable subspace} at time $\tau$ of the pair $(\sg, \opc)$.
If $T = +\infty$, let
\begin{equation}\label{defobssapceinf}
    \Obsspace = \overline{\bigcup_{\tau>0} \Obsspace_\tau}.
\end{equation}
be the \emph{observable subspace} of the pair $(\sg, \opc)$.
\end{definition}
The sequence $(\Obsspace_\tau)_{\tau>0}$ is a non-decreasing sequence of closed linear subspaces. Hence,
$\Obsspace = \overline{\lim_{\tau\cv+\infty}\Obsspace_\tau}$, and it may be seen as the observable subspace in infinite time of the pair $(\sg, \opc)$.

Our results rely on a weak detectability hypothesis defined as follows.
\begin{definition}\label{weakdet}
The pair $((A(t))_{t\in[0,T]}, C)$ is said to be $\mu$-weakly detectable for some $\mu\geq0$ if for all $t\in[0, T]$, 
\begin{equation}\label{eq_WeaklyDetec}
\psX{\opa(t) x}{x} \le \mu  \normY{\opc x}^2,\qquad \forall x\in\doma.
\end{equation}
\end{definition}
We now state our main results about the convergence of the asymptotic observer and the back and forth observer.
In general, the convergence holds only in the weak topology.
\subsection{Weak asymptotic observer}

\begin{theorem}\label{thmain}
Assume that $T = +\infty$ and $((A(t))_{t\geq0}, C)$ is $\mu$-weakly detectable and $r>\mu$.
Assume that there exist an increasing positive sequence $(t_n)_{n\geq0}\cv+\infty$ and an evolution system $(\sg_\infty(t, s))_{0\leq s\leq t}$ on $\XX$ such that for all $\tau\geq0$,
\begin{equation}\label{E:conv-T-to-Tinfinity}
\|\sg(t_n+t, t_n)-\sg_\infty(t, 0)\|_{\linx} \cv0 \text{~as~} n\cv+\infty
\text{~uniformly in~} t\in[0, \tau],
\end{equation}
Let $\Obsspace$ be
the observable subspace of the pair $(\sg_\infty, \opc)$.
Then for all $\eps_0\in\XX$,
\begin{equation}\label{eq_obs_conv}
    \pobs\sgeps(t_n, 0)\eps_0\cvfl{n\cv+\infty}0.
\end{equation}
Moreover, if $(t_{n+1}-t_n)_{n\geq0}$ is bounded and $\Obsspace = \XX$, then
\eqref{obs} is a weak asymptotic observer of~\eqref{syst}.
\end{theorem}
The proof of Theorem~\ref{thmain} is given in Section~\ref{secproofthmain}.
In the autonomous context,
every increasing positive sequence $(t_n)_{n\geq0}\cv+\infty$ is such that
$\sg(t_n+t, t_n) = \sg(t)$ for all $t\geq0$.
Hence~\eqref{eq_obs_conv} holds for all such sequence $(t_n)_{n\geq0}$ and with $\Obsspace$ the observable subspace of $(\sg, C)$. This leads to the following corollary.

\begin{corollary}\label{cormain}
Suppose that~\eqref{syst} is autonomous, $(A, C)$ is $\mu$-weakly detectable and $r>\mu$.
Let $\Obsspace$
be the observable subspace of $(\sg, \opc)$.
Then, \eqref{obs} is a weak asymptotic $\Obsspace$-observer of~\eqref{syst}.
\end{corollary}

\subsection{Weak back and forth observer}

\begin{theorem}\label{thmainbf}
Assume that $T < +\infty$
and $(\sg(t, s))_{0\leq s, t\leq T}$ is a bi-directional evolution system.
Suppose that both $((A(t))_{t\in[0, T]}, C)$ and $((-A(t))_{t\in[0, T]}, C)$ are $\mu$-weakly detectable and $r>\mu$.
Let $\Obsspace_T$ be the observable subspace at time $T$ of the pair $(\sg, \opc)$.
Then,
the system \eqref{obs2}-\eqref{obs2b}
is a weak back and forth $\Obsspace_T$-observer of~\eqref{syst}.
\end{theorem}
The proof of Theorem~\ref{thmainbf} is given in Section~\ref{secproofthmainbf}.
Under additional assumptions on the system, the strong convergence of the observers holds.

\subsection{Strong asymptotic observer}

\begin{theorem}\label{thmainstr}
Assume that $T = +\infty$.
Suppose that there exists $\tau>0$ such that $t\mapsto A(t)$ is $\tau$-periodic.
Let $\Obsspace_\tau$ be
the observable subspace at time $\tau$ of the pair $(\sg, \opc)$.
\begin{enumerate}[label = (\roman*)]
    \item\label{item1}
    Suppose that $((A(t))_{t\geq0}, C)$ is $\mu$-weakly detectable and $r>\mu$.
    Assume that $\sgeps(\tau, 0)$ is normal and bounded from below. If $\Obsspace_\tau=\XX$,
then \eqref{obs} is a strong asymptotic observer of~\eqref{syst}.
    \item\label{item2}
    Suppose that $A(t)$ is skew-adjoint for all $t\in\R_+$ and $\sgeps(\tau, 0)$ is normal.
    If $\sg(t, 0)\Obsspace_\tau \subset \Obsspace_\tau$
    and $\sg(t, 0)\Obsspace_\tau^\perp \subset \Obsspace_\tau^\perp$
    for all $t\in[0, \tau]$,
then \eqref{obs} is a strong asymptotic $\Obsspace_\tau$-observer of~\eqref{syst} for all $r>0$.
\end{enumerate}
\end{theorem}
The proof of Theorem~\ref{thmainstr} is given in Section~\ref{secproofthmainstr}.

\subsection{Strong back and forth observer}

\begin{theorem}\label{thmainbfstr}
Assume that $T < +\infty$
and $(\sg(t, s))_{0\leq s, t\leq T}$ is a bi-directional evolution system.
Let $\Obsspace_T$ be the observable subspace at time $T$ of the pair $(\sg, \opc)$.
\begin{enumerate}[label = (\roman*)]
    \item Suppose that both $((A(t))_{t\in[0, T]}, C)$ and $((-A(t))_{t\in[0, T]}, C)$ are $\mu$-weakly detectable and $r>\mu$.
    Assume that $\sgepsii(0, T)=\sgepsi(T, 0)^*$.
    If $\Obsspace_T=\XX$, then the system \eqref{obs2}-\eqref{obs2b} is a strong back and forth observer of~\eqref{syst}.
    \label{notskew}
    \item\label{itemGH} \cite[Theorem 1.1.2]{haine2014recovering} In the autonomous context, if $A$ is skew-adjoint,
then \eqref{obs} is a strong back and forth $\Obsspace_T$-observer of~\eqref{syst} for all $r>0$.
\end{enumerate}
\end{theorem}
The proof of Theorem~\ref{thmainbfstr} is given in Section~\ref{secproofthmainbfstr}.

\section{Discussion on the results}\label{secdisc}

\subsection{About observability}

For infinite-dimensional systems, there are several observability concepts that are not equivalent (see, \eg \cite[Chapter 6]{TW} in the autonomous context), contrary to the case of finite-dimensional systems.
In particular, one can distinguish the two following main concepts.

\begin{definition}[Exact observability]
    The pair
    $((A(t))_{t\in[0, T]}, C)$
    is said to be exactly observable on $(t_0,t_0+\tau)\subset[0, T]$ if there exists $\delta>0$ such that
    \begin{equation}
        \label{eq_exact_obs}
        \psX{\gram(t_0, \tau)\etat_0}{\etat_0} \geq \delta \norm{z_0}^2,\qquad\forall z_0\in\XX.
    \end{equation}
\end{definition}

\begin{definition}[Approximate observability]
The pair
$((A(t))_{t\in[0, T]}, C)$ is said to be approximately observable on $(t_0,t_0+\tau)\subset[0, T]$ if $\gram(t_0, \tau)$ is injective.
\end{definition}

Clearly, the exact observability of a pair on some time interval implies its 
approximate observability, and the concepts are equivalent in finite-dimension.
The approximate observability in time $\tau$ is equivalent to the fact that $\Obsspace_\tau$, the observable subspace in time $\tau$ of $(\sg, \opc)$, is equal to the whole state space $\XX$.
Our results focus on approximate observability-like assumptions, since the exact observability has already been deeply investigated for both the asymptotic observer and the BFN algorithm (see \eg \cite{ito2011time, Ramdani}).
When the observable subspace is not the full state space, the observers reconstruct only the observable part of the state.

\subsection{About weak detectability}

\begin{remark}
A pair $((A(t))_{t\geq0}, C)$ is said to be \emph{detectable}  if for all pairs of trajectories $(\etat_1, \etat_2)$ of \eqref{syst}, if $Cz_1(t) = Cz_2(t)$ for all $t\geq0$, then $(z_1(t) - z_2(t)) \cv 0$ as $t\cv+\infty$.
This definition is equivalent to the usual definition of detectability in finite-dimension. However, several definitions may be chosen in infinite-dimension, that are all equivalent in finite-dimension.
In this remark, we show how \eqref{eq_WeaklyDetec} may be seen as a weak detectability hypothesis.
Let $((A(t))_{t\in[0, T]}, C)$ be $\mu$-weakly detectable for some $\mu\geq0$.
Then Lemma~\ref{lemcont}, that is proved in Section~\ref{secproofmain}, states that $\sgeps$ is a contraction evolution system,\ie $\|\sgeps(t, s)\|_{\linx}\leq1$ for $0\leq s\leq t\leq T$.
Consider $(\etat_1, \etat_2)$ two trajectories of \eqref{syst} such that $\opc\etat_1(t) = \opc\etat_2(t)$ for all $t\in[0, T]$.
Then $\etat_1$ and $\etat_2$ are also trajectories of \eqref{obs},
and $\etat_1-\etat_2$ is a trajectory of~\eqref{eps}.
Therefore, for all $0\leq s\leq t\leq T$,
\begin{align*}
    \norm{\etat_1(t)-\etat_2(t)}
    = \norm{\sgeps(t, s)(\etat_1(s)-\etat_2(s))}
    \leq \norm{\etat_1(s)-\etat_2(s)}.
\end{align*}
Hence, $[0, T]\ni t\mapsto\norm{\etat_1(t)-\etat_2(t)}$ is non-increasing. This property is indeed weaker than the usual detectability hypothesis, which would state that $\norm{\etat_1(t)-\etat_2(t)}$ tends to $0$ as $t$ goes to infinity.

\end{remark}

\begin{remark}
When stating that a pair  $((A(t))_{t\in[0, T]}, C)$ is $\mu$-weakly detectable, we actually state that the pair is \emph{uniformly} weakly detectable, in the sense that the detectability constant $\mu$ is independent of the time $t\in[0, T]$. Therefore, this assumption is stronger than the weak detectability of each pair $(A(t), C)$ for $t\in[0, T]$. However, if $T<+\infty$ or $t\mapsto A(t)$ is periodic, then the two statements are equivalent, due to the continuity of $[0, T]\ni t\mapsto A(t)x$ for all $x\in\doma$.
\end{remark}

\begin{remark}\label{remdiss}
If $\opa(t)$ is a \emph{dissipative} operator for all $t\in[0, T]$, that is,
\begin{equation}
    \psX{A(t)x}{x} \leq 0,\qquad \forall t\in[0, T],
\end{equation}
then the pair $((A(t))_{t\in[0, T]}, C)$ is $0$-weakly detectable for any output operator $\opc\in\linxy$.
This assumption is the one usually made in the literature to prove the weak convergence of a Luenberger-like observer in infinite-dimension (see, \eg \cite{slemrod1974note, Celle, xu1995observer}).
Therefore, the weak detectability hypothesis may be seen as a weakening of the dissipativity hypothesis, relying on the output operator.
\end{remark}

\begin{remark}\label{rmP}
If there exist a bounded self-adjoint operator $P\in\linx$, $\alpha>0$ and $\mu\geq0$ such that
\begin{equation}\label{Lyap}
\psX{x}{Px}\geq p\norm{x}^2, \quad \psX{Px}{A(t)x}\leq \mu\normY{Cx}^2, \quad \forall x\in \doma,\ \forall t\in[0, T],
\end{equation}
then
the pair $((A(t))_{t\in[0, T]}, C)$ is $\mu$-weakly detectable
provided one endows the Hilbert space $\XX$ with the inner product $\psX{P \cdot}{\cdot}$.
Note that in this case the operator $C^*$ is the adjoint of $C\in\linxy$ with respect to this new inner product,
\ie $\psY{C\cdot}{\cdot} = \psX{P \cdot}{C^*\cdot}$.
Actually, if $\XX$ is finite-dimensional, the existence of $P$ (which is then a positive-definite matrix) such that \eqref{Lyap} holds is a necessary condition for the existence of an asymptotic observer.
\end{remark}

\begin{remark}
In the context of BFN, we require that
$((A(t))_{t\in[0, T]}, C)$ and $((-A(t))_{t\in[0, T]}, C)$ are $\mu$-weakly detectable.
This is equivalent to state that
\begin{equation}
\abs{\psX{\opa(t) x}{x}} \le \mu  \normY{\opc x}^2,\qquad \forall x\in\doma.
\end{equation}
Note that the considered inner product on $\XX$ is the same for both the forward and the backward observer.
If one must change the inner product with a self-adjoint operator $P$ as in Remark~\ref{rmP}, then this change must be done for both observers.
In \cite{hoang:hal-01104899}, the authors proved in the autonomous finite-dimensional context the existence of such a common operator $P$ for both $A$ and $-A$, but the question remains open in infinite-dimension.
\end{remark}

\begin{remark}
The parameter $r>0$ is the observer gain.
If $\opa(t)$ is a \emph{dissipative} operator for all $t\in[0, T]$, then the convergence results hold for all gain $r>0$. Otherwise, the gain must be chosen high enough in order to make up the lack of dissipativity, which is replaced by weak detectability.
Obviously, if a pair is $\mu$-weakly detectable for some $\mu\geq0$, then it is also $\lambda$-weakly detectable for all $\lambda\geq\mu$. This class of observer is what is called \textit{observers with infinite gain margin} since $r$ can be taken as large as requested.
\end{remark}

\subsection{About the results}

\begin{remark}
Our results are linked with the existing literature in the following way.
If $A(t) = A+\sum_{i=1}^pu_i(t)B_i$ where $A, B_0,\dots,B_p$ are skew-adjoint generators of unitary groups on $\XX$ and $u_1,\dots,u_p$ are bounded, then
Theorem~\ref{thmain} is an extension of \cite[Theorem 7]{Celle} to the case where the system is not approximately observable in some finite time. The proofs of Theorems~\ref{thmain} and \ref{thmainbf} follow the path of this seminal paper.
In the autonomous context, we recover the usual weak asymptotic observer in Corollary~\ref{cormain}.
Theorem~\ref{thmainbf} states that only weak convergence of the BFN algorithm holds in general.
Following the way paved by G. Haine in \cite{haine2014recovering}, we prove in Theorem~\ref{thmainbfstr} that the convergence is actually strong under some additional assumptions.
We recall and extend \cite[Theorem 1.1.2]{haine2014recovering} in Theorem~\ref{thmainbfstr}. In particular, we consider non-autonomous systems and do not necessarily assume that $A(t)$ is skew-adjoint for all $t\in[0, T]$.
Moreover, we adapt this technique to the usual asymptotic observer to prove the strong convergence in the case of periodic systems in Theorem~\ref{thmainstr}.
We do not investigate any exact observability-like assumptions, since \cite{slemrod1972linear, liu1997locally, urquiza2005rapid} and \cite{Ramdani, ito2011time} solved the question, at least in the autonomous case, in the asymptotic context and back and forth context respectively.
\end{remark}

\begin{remark}
In Theorem~\ref{thmain}, one of the hypotheses is the existence of an increasing positive sequence $(t_n)_{n\geq0}\cv+\infty$ and an evolution system $(\sg_\infty(t, s))_{0\leq s\leq t}$ on $\XX$ such that $\|\sg(t_n+t, t_n)-\sg_\infty(t, 0)\|_{\linx} \cv0$ as $n\cv+\infty$
uniformly in $t\in[0, \tau]$
for all $\tau\geq0$.
Checking this hypothesis may be a difficult task in general.
However, \cite[Theorem 10.2]{ito} states sufficient conditions on the family of generators $(A(t))_{t\geq0}$ for the existence of such a sequence. 
In Section~\ref{sectrans}, we show how to check this property on a time-varying one-dimensional transport equation with periodic boundary conditions.
\end{remark}

\begin{remark}\label{rmcca}
One of the steps of the proof of Theorem~\ref{thmain} (see Section~\ref{secproofmain}) is to show that for all $\eps_0\in\doma$, $\eps:t\mapsto\sgeps(t, 0)\eps_0$ satisfies
\begin{equation}
    \int_{t_0}^{t_0+\tau} \normY{C\eps(t)}^2\diff t
    \cvl{t_0\cv+\infty}0,\qquad\forall\tau\geq0.
\end{equation}
This does not yields \emph{a priori} that $\opc\eps(t)\cv0$ as $t$ goes to infinity.
However, if there exists a positive constant $\alpha>0$ such that for all $t\geq0$,
\begin{equation}\label{eqCCA}
    \norm{C^*CA(t)x} \leq \alpha \norm{x},
\end{equation}
then
$\opc\eps(t)\cvl{t\cv+\infty}0$. 
Indeed, \eqref{eqintC} will yield
\begin{align}
    \int_0^{+\infty} \normY{C\eps(t)}^2\diff t < + \infty.
\end{align}
Moreover, for all $t\geq0$,
\begin{align*}
\frac12    \frac{\dd}{\dd t}\normY{C\eps(t)}^2
    &=  \psY{C\eps(t)}{C\dot\eps(t)}\\
    &=  \psY{C\eps(t)}{CA\eps(t)} - r \psY{C\eps(t)}{CC^*C\eps(t)}\\
    &= \psX{\eps(t)}{C^*CA\eps(t)} - r \norm{C^*C\eps(t)}^2\\
    &\leq \alpha \norm{\eps_0}^2
\end{align*}
since $\sgeps(t, 0)$ is proved to be a contraction in Lemma~\ref{lemcont}.
Thus, $\normY{C\epsilon}^2$ is an integrable positive function, with bounded derivated. Hence, according to Barbalat's lemma, $\normY{C\epsilon(t)}^2 \cv 0$ as $t\cv+\infty$.

A similar result (with a similar proof) hold for the BFN algorithm. Assume that all the hypotheses
of Theorem~\ref{thmainbf} hold.
If $C^*CA$ is bounded as an operator from $(\doma, \norm{\cdot})$ to $(\XX, \norm{\cdot})$, then
$\opc\eps^{2n}(0)\cv0$ as $n\cv+\infty$.
\end{remark}

\section{Proofs of the results}\label{secproof}

This section is devoted to the proofs of the results stated in Section~\ref{secmain}.
The following remark allows us to reformulate the weak convergence results.

\begin{remark}
For any closed linear subspace $\Obsspace$ of $\XX$ and any sequence $(x_n)_{n\geq0}$ in $\XX$, recall that
$\pobs x_n \cvf0$ as $n\cv+\infty$ if and only if, for all $\psi\in\XX$, 
$\psX{\pobs x_n}{\psi}\cv0$.
As an orthogonal projection, $\pobs$ is a self-adjoint operator, \ie $\pobs = \pobs^*$, and $\ran\pobs=\Obsspace$.
Hence, $\pobs x_n \cvf0$ as $n\cv+\infty$ if and only if, for all $\psi\in\Obsspace$, 
$\psX{\pobs x_n}{\psi}\cv0$.
\end{remark}
All the weak convergence results are proved in the following in accordance with this remark. For example, to prove that \eqref{obs} is a weak asymptotic $\Obsspace$-observer, we prove that 
$\psX{\pobs \sgeps(t, 0)\eps_0}{\psi}\cv0$ as $t\cv+\infty$ for all $\eps_0\in\XX$ and all $\psi\in\Obsspace$. We proceed similarly in the back and forth context.
\begin{lemma}\label{lemdoma}
Let $(\opl_n)_{n\in\N}$ be a bounded sequence in $\linx$,
\ie such that
$\sup_{n\in\N}\|\opl_n\|_{\linx}
\leq \bornL$
for some $\bornL>0$.
Let $U, V \subset X$.
\begin{enumerate}[label = (\roman*)]
\item If
\begin{align*}
    \opl_n\eps_0\cvl{n\cv+\infty}0,\qquad
    \forall\eps_0\in U
\end{align*}
then
\begin{align*}
    \opl_n\eps_0\cvl{n\cv+\infty}0,\qquad
    \forall\eps_0\in\overline{U}.
\end{align*}
\label{itemstr}
    \item If
\begin{align*}
    \psX{\opl_n\eps_0}{\psi}\cvl{n\cv+\infty}0,\qquad
    \forall\eps_0\in U,\quad
    \forall\psi\in V,
\end{align*}
then
\begin{align*}
    \psX{\opl_n\eps_0}{\psi}\cvl{n\cv+\infty}0,\qquad
    \forall\eps_0\in\overline{U},\quad
    \forall\psi\in\overline{V}.
\end{align*}
\label{itemw}
\end{enumerate}
\end{lemma}

\begin{proof}[Proof of \ref{itemstr}]
Let $\bornL$ be a bound of the sequence $(\opl_n)_{n\in\N}$ in $\linx$.
Let $\eps_0\in \overline{U}$ and $\eta>0$.
Then there exists $\tilde{\eps}_0\in U$ such that
$\norm{\eps_0 - \tilde{\eps}_0}\leq \eta
$.
Moreover, there exists $N\in\N$ such that for all $n\geq N$,
$    \norm{\opl_n\tilde{\eps}_0}
    \leq \eta
$.
Then, for all $n\geq N$,
\begin{align*}
    \norm{\opl_n\eps_0}
    \leq \norm{\opl_n\tilde{\eps}_0}
    +  \bornL \norm{\tilde{\eps}_0-\eps_0}
    \leq (1+\bornL)\eta
\end{align*}
since $\|\opl_n\|_{\linx}\leq \bornL$.
Hence $\opl_n\eps_0\cv0$ as $n\cv+\infty$.
\end{proof}

\begin{proof}[Proof of \ref{itemw}]
Let $\eps_0\in \overline{U}$, $\psi\in \overline{V}$ and $\eta>0$.
Then there exist $\tilde{\eps}_0\in U$ and $\tilde{\psi}\in V$ such that
$\norm{\eps_0 - \tilde{\eps}_0}\leq \eta$
and $\norm{\psi - \tilde{\psi}}\leq \eta$.
Moreover, there exists $N\in\N$ such that for all $n\geq N$,
$    \abs{\psX{\opl_n\tilde{\eps}_0}{\tilde{\psi}}}
    \leq \eta
$.
Then, for all $n\geq N$,
\begin{align*}
    \abs{\psX{\opl_n\eps_0}{\psi}}
    &\leq
    \abs{\psX{\opl_n\tilde{\eps}_0}{\tilde{\psi}}}
    +
    \abs{\psX{\opl_n(\eps_0-\tilde{\eps}_0)}{\tilde{\psi}}}\\
    &\quad+
    \abs{\psX{\opl_n\tilde{\eps}_0}{\psi-\tilde{\psi}}}
    +
    \abs{\psX{\opl_n(\eps_0-\tilde{\eps}_0)}{\psi-\tilde{\psi}}}\\
    &\leq \left(1+ \bornL \norm{\tilde{\psi}}+ \bornL \norm{\tilde{\eps}_0} + \bornL \eta\right)\eta
\end{align*}
since $\|\opl_n\|_{\linx}\leq \bornL$.
Hence $\psX{\opl_n\eps_0}{\psi}\cv0$ as $n\cv+\infty$.
\end{proof}

\begin{remark}
An operator $\opl\in\linx$ is said to be a contraction if
$\|\opl\|_{\linx}\leq1$.
If $(\opl_n)_{n\in\N}$ is a sequence of contractions in $\linx$, then it is uniformly bounded by $1$, hence Lemma~\ref{lemdoma} does apply.
In the paper, we use Lemma~\ref{lemdoma} only on sequences of contractions.
\end{remark}

\subsection{Proof of Theorem~\ref{thmain}}
\label{secproofthmain}

\label{secproofmain}
The proof relies on the two following lemmas. The first one shows how the weak detectability is used in the proof, while the second one states a continuity property of the observability Gramian.
We adapt the steps of the proof of \cite[Theorem 7]{Celle}.
In this section, assume that $T=+\infty$.
\begin{lemma}\label{lemcont}
If $((A(t))_{t\geq0}, C)$ is $\mu$-weakly detectable and $r>\mu$, then $\sgeps$ is a contraction evolution system, that is,
\begin{equation}
\|\sgeps(t, s)\|_{\linx} \leq1,\qquad \forall t\geq s\geq 0.
\end{equation}
\end{lemma}
\begin{proof}
Since $\doma$ is dense in $\XX$, it is sufficient to show that
\begin{equation}\label{eqcont}
    \norm{\sgeps(t, t_0) \eps_0} \leq \norm{\eps_0}
\end{equation}
for all $\eps_0\in\doma$ and all $t\geq t_0\geq0$.
Let $t_0\geq0$, $\eps_0\in\doma$ and set $\eps(t) = \sgeps(t, t_0) \eps_0$ for all $t\geq t_0$. Then $\eps\in C^1([0, +\infty), X)$ and for all $t\geq t_0$,
\begin{align}
\frac12\frac{\dd}{\dd t}\norm{\epsilon(t)}^2
&= \psX{\epsilon(t)}{\dot{\epsilon}(t)}\nonumber\\
&= \psX{\epsilon(t)}{\opa(t)\epsilon(t)} - r \psX{\epsilon(t)}{\opc^*\opc\epsilon(t)}\nonumber\\
&\leq -(r-\mu) 
\normY{C\epsilon(t)}^2
\quad
\text{(since $((A(t))_{t\geq0}, C)$ is $\mu$-weakly detectable)}
\label{epsnoninc}
\\
&\leq0\nonumber
\end{align}
since $r>\mu$. Hence $[t_0, +\infty)\ni t \mapsto \norm{\epsilon(t)}^2$ is non increasing, which yields \eqref{eqcont} since $\eps(t_0) = \eps_0$.
\end{proof}

\begin{lemma}\label{lemconv}
If there exist an increasing positive sequence $(t_n)_{n\geq0}\cv+\infty$ and an evolution system $(\sg_\infty(t, s))_{0\leq s\leq t}$ on $\XX$ such that $\|\sg(t_n+t, t_n)-\sg_\infty(t, 0)\|_{\linx} \cv0$ as $n\cv+\infty$ for all $t\geq0$,
then
$\|\gram(t_n, \tau)-\gram_\infty(0, \tau)\|_{\linx}\cv0$ as $n\cv+\infty$.
\end{lemma}
\begin{proof}
For all $\etat_0\in\XX$,
\begin{multline*}
    \norm{(\gram(t_n, \tau)-\gram_\infty(0, \tau))\etat_0}\\
    \begin{aligned}
    &\leq
    \int_0^\tau
    \|\sg(t_n+t, t_n)^*\opc^*\opc\sg(t_n+t, t_n)
    -
    \sg_\infty(t, 0)^*\opc^*\opc\sg_\infty(t, 0)\|_{\linx}
    \norm{\etat_0}\dd t\\
    &\leq
        \tau \norm{\etat_0} \sup_{t\in[0, \tau]} \|\sg(t_n+t, t_n)^*\opc^*\opc\sg(t_n+t, t_n)
    -
    \sg_\infty(t, 0)^*\opc^*\opc\sg_\infty(t, 0)\|_{\linx}.
    \end{aligned}
\end{multline*}
For all $t\in[0, \tau]$,
\begin{multline*}
    \|\sg(t_n+t, t_n)^*C^*C\sg(t_n+t, t_n)
    -
    \sg_\infty(t, 0)^*C^*C\sg_\infty(t, 0)\|_{\linx}\\
    \begin{aligned}
    &\leq \|(\sg(t_n+t, t_n)- \sg_\infty(t, 0))^*\|_{\linx}\|C^*C\sg(t_n+t, t_n) \|_{\linx}\\
    &\quad+
    \|\sg_\infty(t, 0)^*C^*C\|_{\linx}\|\sg(t_n+t, t_n)-\sg_\infty(t, 0)\|_{\linx}\\
    &\leq \|\sg(t_n+t, t_n)-\sg_\infty(t, 0)\|_{\linx}
    \|C\|_{\linxy}^2\big(\|\sg(t_n+t, t_n) \|_{\linx}
    \end{aligned}\\
    \begin{aligned}
    &+\|\sg_\infty(t, 0)\|_{\linx}\big)
    \end{aligned}
\end{multline*}
Recall that $\|\sg(t_n+t, t_n) \|_{\linx}\leq Me^{\omega t}$ by \eqref{E:bound-M-omega} and that \eqref{E:conv-T-to-Tinfinity} implies
$\|\sg(t_n+t, t_n) \|_{\linx}\cv \|\sg_\infty(t, 0) \|_{\linx}$ as $n\to+\infty$. Hence, we also have
$\|\sg_\infty(t, 0) \|_{\linx}\leq Me^{\omega t}$.
Thus,
\begin{multline*}
\|\sg(t_n+t, t_n)^*C^*C\sg(t_n+t, t_n)
-
\sg_\infty(t, 0)^*C^*C\sg_\infty(t, 0)\|_{\linx}
\\
\leq 2\|C\|_{\linxy}^2Me^{\omega t}\|\sg(t_n+t, t_n)-\sg_\infty(t, 0)\|_{\linx}.
\end{multline*}
Hence,
according to \eqref{E:conv-T-to-Tinfinity},
$\|\gram(t_n, \tau)-\gram_\infty(0, \tau)\|_{\linx}\cv0$ as $n\cv+\infty$.
\end{proof}

\begin{proof}[Proof of Theorem~\ref{thmain}]
According to Lemma~\ref{lemcont}, $\sgeps$ is a contraction evolution system. Hence, applying Lemma~\ref{lemdoma}~\ref{itemw} with $\opl_n = \sgeps(t_n, 0)$ for $n\in\N$, it is sufficient to show~\eqref{eq_obs_conv} for all $\psi\in\cup_{\tau\geq0}(\ker\gram_\infty(0, \tau))^\perp$ and all $\eps_0\in\doma$ since $\doma$ is dense is $\XX$.
Let $\eps_0\in\doma$ and set $\eps(t) = \sgeps(t, 0)\eps_0$ for all $t\geq0$.
Since $\sgeps$ is a contraction, $\norm{\eps}$ is non-increasing and whence converges to a finite limit.
Equation~\eqref{epsnoninc} yields for all $t_0, \tau\geq 0$,
\begin{equation}\label{eqintC}
    \int_{t_0}^{t_0+\tau} \normY{C\eps(t)}^2\diff t
    \leq \frac{1}{2(r-\mu)}\left(\norm{\eps(t_0)}^2-\norm{\eps(t_0+\tau)}^2\right).
\end{equation}
Hence,
\begin{equation}
    \int_{t_0}^{t_0+\tau} \normY{C\eps(t)}^2\diff t
    \cvl{t_0\cv+\infty}0.
\end{equation}
According to the Duhamel's formula, for all $t\geq t_0\geq 0$,
\begin{align}\label{duhamel}
\epsilon(t)
= \sg(t, t_0)\epsilon(t_0) 
- r\int_{t_0}^{t} \sg(t, s) \opc^*\opc \epsilon(s) \dd s.
\end{align}
Then
\begin{align*}
\gram(t_0, \tau)\epsilon(t_0)
&=\int_{t_0}^{t_0+\tau}\evol(t, t_0)^*\opc^*\opc\evol(t, t_0)\epsilon(t_0)\dd t\\
&=\int_{t_0}^{t_0+\tau}\evol(t, t_0)^*\opc^*\opc\epsilon(t)\dd t\\
&\qquad+ r \int_{t_0}^{t_0+\tau}\evol(t, t_0)^*\opc^*\opc
\int_{t_0}^{t} \sg(t, s) \opc^*\opc \eps(s) \dd s \dd t.
\end{align*}
By \eqref{E:bound-M-omega} and because $\opc$ is bounded, we have
\begin{align*}
\norm{\gram(t_0, \tau)\epsilon(t_0)}
&\leq
M e^{\omega\tau}
\|\opc\|_{\linxy}
\int_{t_0}^{t_0+\tau} \normY{C\eps(t)}\diff t
\\
&\qquad
+ r \tau M^2 e^{2\omega\tau}
\|\opc\|_{\linxy}^3
\int_{t_0}^{t_0+\tau} \normY{C\eps(t)}\diff t.
\end{align*}
Hence
\begin{equation}\label{gram000}
\gram(t_0, \tau)\epsilon(t_0) \cvl{t_0\cv+\infty}0,\qquad \forall \tau\geq0.
\end{equation}

Now, let $(t_n)_{n\geq0}$ and $(\sg_\infty(t, s))_{0\leq s\leq t}$ be as in the hypotheses of Theorem~\ref{thmain}.
Let $\Omega$ be the set
of limit points of $(\epsilon(t_n))_{n\geq0}$ for the weak topology of $\XX$, that is, the set of points $\xi\in\XX$ such that there exists a subsequence $(n_k)_{k\geq0}$ such that $\eps(t_{n_k})\cvf\xi$ as $k\cv+\infty$.
Since $\epsilon$ is bounded in $\XX$ (because $\sgeps$ is a contraction), by Kakutani's theorem (see, \eg \cite[Theorem 3.17]{Brezis}), the set $\{\epsilon(t_n), n\in\N\}$ is relatively weakly compact in $\XX$. Hence $\Omega$ is not empty.
Let $\xi\in\Omega$ and $(\epsilon(t_{n_k}))_{k\geq0}$ be a subsequence converging weakly to $\xi$.
Then, according to \eqref{gram000} and Lemma~\ref{lemconv},
\begin{align*}
    \norm{\gram_\infty(0, \tau)\eps(t_{n_k})}
    &\leq
    \norm{\gram(t_{n_k}, \tau)\eps(t_{n_k})}\\
    &\qquad+
    \|\gram_\infty(0, \tau)-\gram(t_{n_k}, \tau)\|_{\linx}\norm{\eps_0}\\
    &\cvl{k\cv+\infty}0.
\end{align*}
Hence $\xi\in\ker\gram_\infty(0, \tau)$.
Thus $\Omega\subset\ker\gram_\infty(0, \tau)$.
Let $\psi\in X$.
By definition of $\Omega$, and since $\eps$ is bounded,
for all $\eta>0$,
there exists $N\in\N$ such that for all $n\geq N$,
there exists $\xi_n\in\Omega$ such that
\begin{align*}
    |\psX{\eps(t_n)-\xi_n}{\psi}| \leq \eta.
\end{align*}
Then, if $\psi\in (\ker\gram_\infty(0, \tau))^\perp$, $\psX{\xi_n}{\psi}=0$ which yields
\begin{align*}
    \abs{\psX{\eps(t_n)}{\psi}}
    \leq \abs{\psX{\eps(t_n)-\xi_n}{\psi}} + \abs{\psX{\xi_n}{\psi}}
    \leq \eta.
\end{align*}
Since this result holds for all $\tau\geq0$,
\begin{align*}
    \psX{\eps(t_n)}{\psi} \cvfl{n\cv+\infty} 0,\qquad\forall\psi\in\bigcup_{\tau\geq0}(\ker\gram_\infty(0, \tau))^\perp.
\end{align*}
This conclude the proof of the first part of Theorem~\ref{thmain}.\\

Now, assume that $((t_{n+1}-t_n))_{n\ge 0}$ is bounded and $\Obsspace = \XX$.
It is sufficient to prove that for all increasing positive sequence $(\tau_k)_{k\geq0}\cv+\infty$, $\eps(\tau_k)\cvf0$ as $k\cv+\infty$.
For all $k\in\N$, let $n_k\in\N$ be such that $t_{n_k}\leq \tau_k< t_{n_k+1}$. Then $s_k = \tau_k - t_{n_k}$ is a non-negative bounded sequence.
Hence, up to an extraction of $(t_n)_{n\geq0}$,
it is now sufficient to prove that $\eps(t_n+s_n)\cvf0$ as $n\cv+\infty$ for all non-negative bounded sequence $(s_n)_{n\geq0}$. Set $\bar{s} = \sup_{n\in\N}s_n$.
For all $\psi\in\XX$,
\begin{align*}
    \abs{\psX{\eps(t_n+s_n)}{\psi}}
    &\leq \abs{\psX{\sg_\infty(s_n, 0)\eps(t_n)}{\psi}}\\
    &\quad +  \|(\sg(t_n+s_n, t_n) - \sg_\infty(s_n, 0))\|_{\linx}\norm{\eps_0}\norm{\psi}\\
    &\qquad + \norm{\eps(t_n+s_n) - \sg(t_n+s_n, t_n)\eps(t_n)}\norm{\psi}.
\end{align*}
By \eqref{E:conv-T-to-Tinfinity}, and because $(s_n)_{n\geq0}$ is bounded, it follows that
\begin{align*}
    \|(\sg(t_n+s_n, t_n) - \sg_\infty(s_n, 0))\|_{\linx}\cvl{n\cv+\infty}0.
\end{align*}
Using \eqref{E:bound-M-omega}, \eqref{duhamel} and the Cauchy-Schwarz inequality
\begin{align*}
    \norm{\eps(t_n+s_n) - \sg(t_n+s_n, t_n)\eps(t_n)}
    &\leq r M e^{\omega \bar{s}} \|C\|_{\linxy} \int_{t_n}^{t_n+\bar{s}}\normY{C\eps(t)}\dd t\\
    &\cvl{n\cv+\infty}0.
\end{align*}
Hence, it remains to prove that $\sg_\infty(s_n, 0)\eps(t_n)\cvf0$ as $n\cv+\infty$.
For all $t\geq0$,
\eqref{E:bound-M-omega} and \eqref{E:conv-T-to-Tinfinity}
yield
$\|\sg_\infty(t, 0)\|_{\linx} \leq Me^{\omega t}$, and thus for $\psi\in\XX$,
\begin{align*}
    \abs{\psX{\sg_\infty(s_n, 0)\eps(t_n)}{\psi}}
    \leq
    Me^{\omega\bar{s}}\norm{\eps_0}\norm{\psi}.
\end{align*}
Let $\ell\in\R$ and $(n_k)_{k\geq0}$ a subsequence such that
$\abs{\psX{\sg_\infty(s_{n_k}, 0)\eps(t_{n_k})}{\psi}}\cv\ell$ as $k\cv+\infty$. We now show that $\ell = 0$ to end the proof.
Since $(s_n)_{n\geq0}$ is bounded
and $s\mapsto\sg_\infty(s, 0)^*\psi$ is continuous in the strong topology of $\XX$,
$(\sg_\infty(s_{n_k}, 0)^*\psi)_{k\geq0}$ converges strongly up to a new extraction of $(s_{n_k})_{k\geq0}$ to some $\xi\in\XX$. Then, for all $k\in\N$,
\begin{align*}
    \abs{\psX{\sg_\infty(s_{n_k}, 0)\eps(t_{n_k})}{\psi}}
    &= \abs{\psX{\eps(t_{n_k})}{\sg_\infty(s_{n_k}, 0)^*\psi}}\\
    &\leq
    \abs{\psX{\eps(t_{n_k})}{\xi}}
    + \norm{\sg_\infty(s_{n_k}, 0)^*\psi - \xi}\norm{\eps_0}\\
    &\cvl{k\cv+\infty}0.
\end{align*}
Thus $\ell = 0$.
\end{proof}

\subsection{Proof of Theorem~\ref{thmainbf}}
\label{secproofthmainbf}

Assume that $T<+\infty$ and $(\sg(t, s))_{0\leq s, t\leq T}$ is a bi-directional evolution system. We adapt the proof of Theorem~\ref{thmain} to the BFN algorithm (see Section~\ref{secproofthmain}).
The lemmas involved and steps of the proof are very similar.

\begin{lemma}\label{lemcontbf}
If both $((A(t))_{t\in[0, T]}, C)$ and $((-A(t))_{t\in[0, T]}, C)$ are $\mu$-weakly detectable and $r>\mu$, then $\sgepsi$ (resp. $\sgepsii$) is a forward (resp. backward) contraction bi-directional evolution system, that is,
\begin{equation}
\|\sgepsi(t, s)\|_{\linx} \leq1
\quad\text{and}\quad
\|\sgepsii(s, t)\|_{\linx} \leq1,
\qquad \forall t\geq s\geq 0.
\end{equation}
\end{lemma}
\begin{proof}
Since $\doma$ is dense in $\XX$, it is sufficient to show that
\begin{equation}
    \norm{\sgepsi(t, t_0) \eps_0} \leq \norm{\eps_0}
    \quad\text{and}\quad
    \norm{\sgepsii(t, t_0) \eps_0} \geq \norm{\eps_0}
\end{equation}
for all $\eps_0\in\doma$ and all $t\geq t_0\geq0$.
Let $t_0\geq0$, $\eps_0\in\doma$ and set $\eps_+(t) = \sgepsi(t, t_0) \eps_0$ and
$\eps_-(t) = \sgepsii(t, t_0) \eps_0$
for all $t\geq t_0$. Then $\eps^i\in C^1([0, +\infty), X)$ for $i\in\{0, 1\}$ and for all $t\geq t_0$,
\begin{align}
\frac12\frac{\dd}{\dd t}\norm{\eps_+(t)}^2
&= \psX{\eps_+(t)}{\dot{\eps}_+(t)}\nonumber\\
&= \psX{\eps_+(t)}{\opa(t)\eps_+(t)} - r \psX{\eps_+(t)}{\opc^*\opc\eps_+(t)}\nonumber\\
&\leq -(r-\mu) 
\normY{C\eps_+(t)}^2
\quad
\text{(since $((A(t))_{t\geq0}, C)$ is $\mu$-weakly detectable)}
\label{epsnonincbf}
\\
&\leq0\nonumber
\end{align}
and
\begin{align}
\frac12\frac{\dd}{\dd t}\norm{\eps_-(t)}^2
&= \psX{\eps_-(t)}{\dot{\eps}_-(t)}\nonumber\\
&= \psX{\eps_-(t)}{\opa(t)\eps_-(t)} + r \psX{\eps_-(t)}{\opc^*\opc\eps_-(t)}\nonumber\\
&\geq (r-\mu) 
\normY{C\eps_-(t)}^2
\quad
\text{(since $((-A(t))_{t\geq0}, C)$ is $\mu$-weakly detectable)}
\label{epsnonincbfii}
\\
&\geq0\nonumber
\end{align}
since $r>\mu$. Hence $[t_0, +\infty)\ni t \mapsto \norm{\eps_+(t)}^2$ is non-increasing and
$[t_0, +\infty)\ni t \mapsto \norm{\eps_-(t)}^2$ is non-decreasing, which yields \eqref{eqcont} since $\eps_+(t_0) = \eps_-(t_0) = \eps_0$.
\end{proof}

\begin{proof}[Proof of Theorem~\ref{thmainbf}]
According to Lemma~\ref{lemcontbf}, $\sgepsi$ (resp. $\sgepsii$) is a forward (resp. backward) contraction bi-directional evolution system.
Let $\opl = \sgepsii(0, T)\sgepsi(T, 0)\in\linx$. Then $\opl^n$ is a contraction for all $n\in\N$.
Hence,
applying Lemma~\ref{lemdoma}~\ref{itemw}, it is sufficient to show
that
$\psX{\opl^n\eps_0}{\psi}\cv0$ as $n\cv+\infty$
for all $\psi\in\cup_{\tau\geq0}(\ker\gram(0, T))^\perp$ and all $\eps_0\in\doma$ since $\doma$ is dense is $\XX$.
Let $\eps_0\in\doma$ and set $\eps^{2n}(t) = \sgepsi(t, 0)\opl^n\eps_0$ for all $t\geq0$ and all $n\in\N$.
Since $\opl$ is a contraction, $\norm{\eps^{2n}(0)}$ is non-increasing
and thus has a finite limit as $n$ goes to infinity.
Moreover, 
\begin{align*}
    \norm{\eps^{2n}(T)}
    = \norm{\sgepsi(T, 0)\opl^n\eps_0}
    &= \norm{\sgepsii(T, 0)\opl^{n+1}\eps_0}\\
    &=\norm{\sgepsii(T, 0)\eps^{2(n+1)}(0)}
    \geq \norm{\eps^{2(n+1)}(0)}.
\end{align*}
Then~\eqref{epsnonincbf} yields for all $n\in\N$
\begin{align*}
    \int_{0}^{T} \normY{C\eps^{2n}(t)}^2\diff t
    &\leq \frac{1}{2(r-\mu)}\left(\norm{\eps^{2n}(0)}^2-\norm{\eps^{2n}(T)}^2\right)\\
    &\leq \frac{1}{2(r-\mu)}\left(\norm{\eps^{2n}(0)}^2-\norm{\eps^{2(n+1)}(0)}^2\right).
\end{align*}
Hence,
\begin{equation}\label{conv-int-C-eps2}
    \int_{0}^{T} \normY{C\eps^{2n}(t)}^2\diff t
    \cvl{n\cv+\infty}0.
\end{equation}
According to the Duhamel's formula, for all $n\in\N$,
\begin{align}\label{duhamel-bf}
\epsilon^{2n}(t)
= \sg(t, 0)\epsilon^{2n}(0) 
- r\int_{0}^{t} \sg(t, s) \opc^*\opc \epsilon^{2n}(s) \dd s.
\end{align}
Then
\begin{align*}
\gram(0, T)\epsilon^{2n}(0)
&=\int_{0}^{T}\evol(t, 0)^*\opc^*\opc\evol(t, 0)\epsilon^{2n}(0)\dd t\\
&=\int_{0}^{T}\evol(t, 0)^*\opc^*\opc\epsilon^{2n}(t)\dd t\\
&\qquad+ r \int_{0}^{T}\evol(t, 0)^*\opc^*\opc
\int_{0}^{t} \sg(t, s) \opc^*\opc \eps^{2n}(s) \dd s \dd t.
\end{align*}
According to \eqref{E:bound-M-omega} and because $\opc$ is bounded,
\begin{align*}
\norm{\gram(0, T)\epsilon^{2n}(0)}
&\leq
M e^{\omega T}
\|\opc\|_{\linxy}
\int_{0}^{T} \normY{C\eps^{2n}(t)}\diff t
\\
&\quad
+ r T M^2 e^{2\omega T}
\|\opc\|_{\linxy}^3
\int_{0}^{T} \normY{C\eps^{2n}(t)}\diff t.
\end{align*}
Hence $\gram(0, T)\epsilon^{2n}(0) \cv0$ as $n\cv+\infty$.\\

Now,
let $\Omega$ be the set
of limit points of $(\epsilon^{2n}(0))_{n\geq0}$ for the weak topology of $\XX$, that is, the set of points $\xi\in\XX$ such that there exists a subsequence $(n_k)_{k\geq0}$ such that $\epsilon^{2n_k}(0)\cvf\xi$ as $k\cv+\infty$.
Since $(\epsilon^{2n}(0))_{n\geq0}$ is bounded in $\XX$ (because $\opl$ is a contraction), by Kakutani's theorem (see, \eg \cite[Theorem 3.17]{Brezis}), the set $\{\epsilon^{2n}(0), n\in\N\}$ is relatively weakly compact in $\XX$. Hence $\Omega$ is not empty.
Let $\xi\in\Omega$ and $(\epsilon^{2n_k}(0))_{k\geq0}$ be a subsequence converging weakly to $\xi$.
Then $\gram(0, T)\xi = 0$ by uniqueness of the weak limit.
Thus $\Omega\subset\ker\gram(0, T)$.
Let $\psi\in X$.
By definition of $\Omega$, and since $(\epsilon^{2n}(0))_{n\geq0}$ is bounded,
for all $\eta>0$,
there exists $N\in\N$ such that for all $n\geq N$,
there exits $\xi_n\in\Omega$ such that
\begin{align*}
    |\psX{\eps^{2n}(0)-\xi_n}{\psi}| \leq \eta.
\end{align*}
Then, if $\psi\in (\ker\gram(0, T))^\perp$, $\psX{\xi_n}{\psi}=0$ which yields
\begin{align*}
    \abs{\psX{\eps^{2n}(0)}{\psi}}
    \leq \abs{\psX{\eps^{2n}(0)-\xi_n}{\psi}} + \abs{\psX{\xi_n}{\psi}}
    \leq \eta,
\end{align*}
\ie
\begin{align*}
    \psX{\eps^{2n}(0)}{\psi} \cvfl{n\cv+\infty} 0,\qquad\forall\psi\in\bigcup_{\tau\geq0}(\ker\gram(0, T))^\perp.
\end{align*}
This ends the proof of Theorem~\ref{thmainbf}.
\end{proof}

\subsection{Proof of Theorem~\ref{thmainstr}}
\label{secproofthmainstr}

Let us first state two important lemmas.
They imply that the dynamics of the error system (\ref{eps2}-\ref{eps2b}) may be decomposed on the two subspaces $\Obsspace_\tau$ and $\Obsspace_\tau^\perp$.
Therefore, the initial estimation of the unobservable part of the system $\proj_{\Obsspace_\tau^\perp}\etath_0$ does not affect the reconstruction of the observable part $\pobst\etat(t)$ at all.
In Statement~\ref{item1}, the hypothesis $\Obsspace_\tau = \XX$ holds, so that these two lemmas are useless.
On contrary, they are used to prove Statement~\ref{item2}.

\begin{lemma}\label{lemstable}
Assume that $T = +\infty$ and $A(t)$ is skew-adjoint for all $t\in\R_+$.
Let $\Obsspace_\tau$ be
the observable subspace at time $\tau$ of the pair $(\sg, \opc)$.
Set $\opl = \sgeps(\tau, 0)^*\sgeps(\tau, 0)$.
Then $\opl \Obsspace_\tau \subset \Obsspace_\tau$
and  $\opl \Obsspace_\tau^\perp \subset \Obsspace_\tau^\perp$.
\end{lemma}

\begin{proof}
According to \cite[Chapter 3, Lemma 1.1]{daleckii2002stability}, since $A(t)$ is skew-adjoint for all $t\in\R$, it is the generator of a unitary bi-directional evolution system, still denoted by $\sg$. In particular, for all $t\geq s\geq t_0\in\R$, $\sg(t, s)^*\sg(t, t_0) = \sg(s, t_0)$.

Let $\eps_0\in \doma\cap\Obsspace_\tau$. For all $\psi_0\in\doma\cap\Obsspace_\tau^\perp = \doma\cap\ker \gram(0, \tau)$, the Duhamel's formula \eqref{duhamel} yields
\begin{align}\label{eq_Lep}
    \psX{L\eps_0}{\psi_0}
    &= \psX{\sgeps(\tau, 0)\eps_0}{\sgeps(\tau, 0)\psi_0}\nonumber\\
    &= \psX{\eps_0}{\sg(\tau, 0)^*\sgeps(\tau, 0)\psi_0}
    - r\int_0^\tau \psX{C\sgeps(s, 0)\eps_0}{C\sg(\tau, s)^*\sgeps(\tau, 0)\psi_0}\dd s.
\end{align}
Since $\psi_0\in\ker \gram(0, \tau)$, $\opc\sg(t, 0)\psi_0=0$ for all $t\in[0, \tau]$.
Set $\psi(t) = \sg(t, 0)\psi_0$
and $\bar{\psi}(t) = \sgeps(t, 0)(-\psi_0)$.
Then $\psi + \bar{\psi}$ is the unique solution of \eqref{obs}
starting from $0\in\dom$ and
with $y(t) = 0$ for all $t\in[0, \tau]$.
Hence, $\psi + \bar{\psi} = 0$ on $[0, \tau]$, \ie
$
\sgeps(t, 0)\psi_0
= \sg(t, 0)\psi_0
$
for all $t\in[0, \tau]$.
Then, \eqref{eq_Lep} yields
\begin{align*}
    \psX{L\eps_0}{\psi_0}
    &= \psX{\eps_0}{\sg(\tau, 0)^*\sg(\tau, 0)\psi_0}
    - r\int_0^\tau \psX{C\sg(s, 0)\eps_0}{C\sg(\tau, s)^*\sg(\tau, 0)\psi_0}\dd s.\\
    &= \psX{\eps_0}{\psi_0}
    - r\int_0^\tau \psX{C\sg(s, 0)\eps_0}{C\sg(s, 0)\psi_0}\dd s.\\
    &=0.
\end{align*}
Thus, since $\dom$ is dense in $\XX$, $L\eps_0\in\Obsspace_\tau$ for all $\eps_0\in\Obsspace_\tau$.
Now, let $\eps_0\in \Obsspace_\tau^\perp$ and $\psi_0\in\Obsspace_\tau$.
Since $\opl$ is self-adjoint,
$\psX{\opl\eps_0}{\psi_0} = \psX{\eps_0}{\opl\psi_0} = 0$ from above.
Hence, $L\eps_0\in\Obsspace_\tau^\perp$.
\end{proof}

\begin{lemma}\label{lemstable2}
Assume that $T = +\infty$ and $A(t)$ is skew-adjoint for all $t\in\R_+$.
Let $\Obsspace_\tau$ be
the observable subspace at time $\tau$ of the pair $(\sg, \opc)$.
If $\sg(t, 0)\Obsspace_\tau \subset \Obsspace_\tau$
and $\sg(t, 0)\Obsspace_\tau^\perp \subset \Obsspace_\tau^\perp$
for all $t\in[0, \tau]$,
then $\sgeps(t, 0)\Obsspace_\tau \subset \Obsspace_\tau$
and $\sgeps(t, 0)\Obsspace_\tau^\perp \subset \Obsspace_\tau^\perp$
for all $t\in[0, \tau]$.
\end{lemma}
\begin{proof}
As in Lemma~\ref{lemstable}, $(A(t))_{t\geq0}$ generates a unitary bi-directional evolution system $\sg$.
Hence, $\sg(t, 0)\Obsspace_\tau \subset \Obsspace_\tau$ if and only if $\sg(0, t)\Obsspace_\tau^\perp \subset \Obsspace_\tau^\perp$
and
$\sg(t, 0)\Obsspace_\tau^\perp \subset \Obsspace_\tau^\perp$ if and only if $\sg(0, t)\Obsspace_\tau \subset \Obsspace_\tau$.
Assume that all these inclusions hold.
Let $t\in\R_+$ and $\eps_0\in\Obsspace_\tau$.
For all $\psi_0\in\Obsspace_\tau^\perp$, the Duhamel's formula \eqref{duhamel} yields
\begin{align*}
    \psX{\sgeps(t, 0)\eps_0}{\psi_0}
    &= \psX{\sg(t, 0)\eps_0}{\psi_0}
    - r\int_0^t \psX{C\sgeps(s, 0)\eps_0}{C\sg(t, s)^*\psi_0}\dd s\\
    &= \psX{\sg(t, 0)\eps_0}{\psi_0}
    - r\int_0^t \psX{C\sgeps(s, 0)\eps_0}{C\sg(s, 0)\sg(0, t)\psi_0}\dd s.
\end{align*}
Since $\sg(t, 0)\Obsspace_\tau \subset \Obsspace_\tau$
and $\sg(0, t)\Obsspace_\tau^\perp \subset \Obsspace_\tau^\perp$,
we get
$\psX{\sg(t, 0)\eps_0}{\psi_0} = 0$ and $C\sg(s, 0)\sg(0, t)\psi_0 = 0$ for all $s\in[0, \tau]$, respectively.
Hence, $\sgeps(t, 0)\Obsspace_\tau \subset \Obsspace_\tau$.
Similarly, if $\eps_0\in\Obsspace_\tau^\perp$, then for all
$\psi_0\in\Obsspace_\tau$, the Duhamel's formula \eqref{duhamel} yields
\begin{align*}
    \psX{\sgeps(t, 0)\eps_0}{\psi_0}
    &= \psX{\eps_0}{\sgeps(0, t)\psi_0}\\
    &= \psX{\eps_0}{\sg(0, t)\psi_0}
    + r\int_0^t \psX{C\sg(0, s)^*\eps_0}{C\sgeps(s, 0)\psi_0}\dd s\\
    &= \psX{\eps_0}{\sg(0, t)\psi_0}
    + r\int_0^t \psX{C\sg(s, 0)\eps_0}{C\sgeps(s, 0)\psi_0}\dd s.
\end{align*}
Since $\sg(0, t)\Obsspace_\tau \subset \Obsspace_\tau$
and $\sg(t, 0)\Obsspace_\tau^\perp \subset \Obsspace_\tau^\perp$,
it holds that
$\psX{\eps_0}{\sg(0, t)\psi_0}=0$ and
$C\sg(s, 0)\eps_0 = 0$ for all $s\in[0, \tau]$, respectively.
Hence, $\sgeps(t, 0)\Obsspace_\tau^\perp \subset \Obsspace_\tau^\perp$.
\end{proof}

\begin{proof}[Proof of Theorem~\ref{thmainstr}]

Let $\tau>0$ be as in the assumptions of the theorem, and set $\opl = \sgeps(\tau, 0)^*\sgeps(\tau, 0)$.
If the hypotheses of Statement~\ref{item1} are satisfied, then
the conclusions of Lemmas~\ref{lemstable} and~\ref{lemstable2} hold (since $\Obsspace_\tau=\XX$).
Otherwise, if the hypotheses of Statement~\ref{item2} are satisfied, then 
the hypotheses and conclusions of Lemmas~\ref{lemstable} and~\ref{lemstable2} hold.
Assume for a moment that $A(t)$ is skew-adjoint for all $t\in\R_+$.
Then $((A(t))_{t\geq0}, C)$ is $0$-weakly dissipative (see Remark~\ref{remdiss}) and $(\sg(t, s))_{t, s\geq0}$ is a unitary bi-directional evolution system (see \cite[Chapter 3, Lemma 1.1]{daleckii2002stability}).
Hence, applying \cite[Chapter 5, Theorem 2.3]{Pazy} to
$(\sg(t, s))_{0\leq s\leq t\leq \tau}$ and $(\sg(\tau-t, \tau-s))_{0\leq s\leq t\leq \tau}$ perturbed with the bounded operators $-rC^*C$ and $rC^*C$ respectively,
we obtain that
$(\sgeps(t, s))_{0\leq s\leq t\leq \tau}$ and $(\sgeps(\tau-t, \tau-s))_{0\leq s\leq t\leq \tau}$ are two evolution systems.
Moreover, the condition $\sgeps(s, t)\sgeps(t, s) = \Id_\XX$ for all $t, s\in \R_+$ is also satisfied, due to the uniqueness of solutions of \eqref{eps}.
Hence, according to the characterization given below Definition~\ref{def:bi},
$(\sgeps(t, s))_{0\leq s\leq t\leq \tau}$ is actually a bi-directional evolution system, that can be naturally extended on $\R_+$.

Thus, both Statements \ref{item1} and \ref{item2} are implied by the following:
\begin{it}
\begin{enumerate}
    \item[(iii)]\label{item3}
    Suppose that $((A(t))_{t\geq0}, C)$ is $\mu$-weakly detectable, $\sgeps(\tau, 0)$ is bounded from below and normal.
    Assume that the conclusions of Lemmas ~\ref{lemstable} and~\ref{lemstable2} are satisfied.
    Then \eqref{obs} is a strong asymptotic $\Obsspace_\tau$-observer of~\eqref{syst} for all $r>\mu$.
\end{enumerate}
\end{it}
Suppose that the assumptions of \emph{(iii)} hold.
We aim to show that $\pobst\sgeps(t, 0)\eps_0\cv0$ as $t\cv+\infty$ for all $\eps_0\in\XX$.
The floor function is denoted by $\lfloor\cdot\rfloor$.
For all $t\in\R_+$ and all $\eps_0\in\XX$,
\begin{align}
    \norm{\pobst\sgeps\left(t, 0\right)\eps_0}
    &= \norm{\pobst\sgeps\left(t, \left\lfloor\frac{t}{\tau}\right\rfloor\tau\right)\sgeps\left(\left\lfloor\frac{t}{\tau}\right\rfloor\tau, 0\right)\eps_0}\nonumber\\
    &= \norm{\pobst\sgeps\left(t-\left\lfloor\frac{t}{\tau}\right\rfloor\tau, 0\right)\sgeps\left(\tau, 0\right)^{\lfloor \frac{t}{\tau} 
    \rfloor}\eps_0}
    \tag{since $t\mapsto A\left(t\right)$ is $\tau$-periodic}\\
    &= \norm{\sgeps\left(t-\left\lfloor\frac{t}{\tau}\right\rfloor\tau, 0\right)\sgeps\left(\tau, 0\right)^{\lfloor \frac{t}{\tau}
    \rfloor}\pobst\eps_0}
    \tag{by the conclusion of Lemma~\ref{lemstable2}}\\
    &\leq \norm{\sgeps\left(t-\left\lfloor\frac{t}{\tau}\right\rfloor\tau, 0\right)}\norm{\sgeps\left(\tau, 0\right)^{\lfloor \frac{t}{\tau}
    \rfloor}\pobst\eps_0}\nonumber\\
    &\leq \norm{\sgeps\left(\tau, 0\right)^{\lfloor \frac{t}{\tau}
    \rfloor}\pobst\eps_0}.
    \tag{according to Lemma~\ref{lemcont}}
\end{align}
Moreover, for all $n\in\N$,
$\psX{\opl^n\eps_0}{\eps_0}
= \norm{\sgeps(\tau, 0)^n\eps_0}^2$
since $\sgeps(\tau, 0)$ is normal.
Thus, applying Lemma~\ref{lemdoma}~\ref{itemstr}, it remains to prove that for all $\eps_0\in\doma\cap\Obsspace_\tau$, $L^n\eps_0\cv0$ as $n\cv+\infty$ since $\doma$ is dense in $\XX$ and $L^n$ is a contraction for all $n\in\N$.

The proof is an adaptation of the strategy developed in \cite[Theorem 1.1.2]{haine2014recovering}. First, we investigate the properties of $\opl$.
It is self-adjoint positive-definite since $\sgeps(\tau, 0)$ is bounded from below.
Let $\eps_0\in\doma\cap\Obsspace_\tau$.
The hypotheses of Lemma~\ref{lemcont} hold. Hence, $\sgeps$ is a contraction evolution system, and \eqref{epsnoninc} yields
\begin{align}\label{psL}
    \psX{\opl \eps_0}{\eps_0}
    = \norm{\sgeps(\tau, 0)\eps_0}^2
    \leq \norm{\eps_0}^2
    - 2(r-\mu)\int_0^\tau\normY{C\sgeps(t, 0)\eps_0}^2\dd t.
\end{align}
Denote by $L^{\frac{1}{2}}$ the square root of $L$.
Then 
\begin{align*}
    \norm{L\eps_0}^2
    &= \psX{LL^{\frac{1}{2}}\eps_0}{L^{\frac{1}{2}}\eps_0}\\
    &\leq \norm{L^{\frac{1}{2}}\eps_0}^2
    - 2(r-\mu)\int_0^\tau\normY{C\sgeps(t, 0)L^{\frac{1}{2}}\eps_0}^2\dd t\\
    &\leq \psX{\opl \eps_0}{\eps_0}\\
    &\leq \norm{\eps_0}^2
    - 2(r-\mu)\int_0^\tau\normY{C\sgeps(t, 0)\eps_0}^2\dd t.
\end{align*}
If $\norm{L\eps_0}=\norm{\eps_0}$,
then $C\sgeps(t, 0)\eps_0 = 0$
for all $t\in[0, \tau]$.
Hence, according to the Duhamel's formula \eqref{duhamel},
$\sgeps(t, 0)\eps_0= \sg(t, 0)\eps_0$ for all $t\in[0, \tau]$.
Then $\gram(0, \tau)\eps_0=0$,
\ie$\eps_0\in\Obsspace_\tau\cap\Obsspace_\tau^\perp=\{0\}$.

Thus, $\norm{L\eps_0}<\norm{\eps_0}$ if $\eps_0\neq 0$.
Moreover, \eqref{epsnoninc} yields for all $\eps_0\in\XX$ and all $n\in\N$
\begin{align*}
    \psX{\opl^{n+1} \eps_0}{\eps_0}
    - \psX{\opl^{n} \eps_0}{\eps_0}
    &= \norm{\sgeps((n+1)\tau, 0)\eps_0}^2
    - \norm{\sgeps(n\tau, 0)\eps_0}^2\\
    &\leq
    - 2(r-\mu)\int_{n\tau}^{(n+1)\tau}\normY{C\sgeps(t, 0)\eps_0}^2\dd t\\
    &\leq 0.
\end{align*}
Then $(L^n)_{n\geq0}$ is a non-increasing sequence of bounded self-adjoint positive-definite operators on the vector space $\Obsspace_\tau$ (by the invariance property).
Hence, according to \cite[Lemma 12.3.2]{TW}, there exists a bounded self-adjoint positive-definite operator $L_\infty\in\lin(\Obsspace_\tau)$ such that $L_\infty\leq L^n$ for all $n\in\N$ and
$L^n\eps_0\cv L_\infty\eps_0$ as $n\cv+\infty$ for all $\eps_0\in\Obsspace_\tau$. It remains to prove that $L_\infty=0$.

For all $x_1, x_2\in\Obsspace_\tau$ and all $n\in\N$,
\begin{align*}
    \psX{L_\infty x_1}{L_\infty x_2}
    &= \psX{L_\infty x_1}{(L_\infty-L^n)x_2}+\psX{(L_\infty-L^n) x_1}{L^n x_2}\\
    &\quad+\psX{L^n x_1}{L^n x_2}.
\end{align*}
Since $L$ is self-adjoint,
\begin{align*}
    \psX{L^n x_1}{L^n x_2}
    = \psX{L^{2n} x_1}{x_2}
    \cvl{n\cv+\infty}\psX{L_\infty x_1}{x_2}.
\end{align*}
Hence $L^2_\infty = L_\infty$.
Moreover, for all $\eps_0\in\Obsspace_\tau\setminus\{0\}$,
\begin{align*}
    \norm{L_\infty\eps_0}^2
    = \psX{L_\infty^2\eps_0}{\eps_0}
    =\psX{L_\infty\eps_0}{\eps_0}
    \leq \psX{L^2\eps_0}{\eps_0}
    = \norm{L\eps_0}^2
    < \norm{\eps_0}^2.
\end{align*}
Hence
$\norm{L_\infty\eps_0}^2
    = \norm{L_\infty^2\eps_0}^2
    <  \norm{L_\infty\eps_0}^2
$ if $L_\infty\eps_0\neq0$.
Thus $L_\infty\eps_0 = 0$ for all  $\eps_0\in\Obsspace_\tau$, which ends the proof.
\end{proof}

\subsection{Proof of Theorem~\ref{thmainbfstr}}
\label{secproofthmainbfstr}
Statement~\ref{itemGH} is a recall of the previous work of \cite{haine2014recovering}. We adapt the method to prove Statement~\ref{notskew}.

\begin{proof}[Proof of Theorem~\ref{thmainbfstr}~\ref{notskew}]
Assume that $T < +\infty$
and $(\sg(t, s))_{0\leq s, t\leq T}$ is a bi-directional evolution system.
Suppose that $((A(t))_{t\in[0, T]}, C)$ and $((-A(t))_{t\in[0, T]}, C)$ are $\mu$-weakly detectable and $r>\mu$.
Assume also that $\Obsspace_T=\XX$ and $\sgepsii(0, T)=\sgepsi(T, 0)^*$.
We follow the same strategy as in the proof of Theorem~\ref{thmainstr} (see Section~\ref{secproofthmainstr}).

Let $\opl = \sgepsii(0, T)\sgepsi(T, 0) = \sgepsi(T, 0)^*\sgepsi(T, 0)$ (as in the proof of Theorem~\ref{thmainbf}, Section~\ref{secproofthmainbf}).
Then, it is sufficient to prove that for all $\eps_0\in\Obsspace_\tau$, $L^n\eps_0\cv$0 as $n\cv+\infty$.
The operator $L$ is self-adjoint positive-definite since 
$\sgepsi(\tau, 0)$
is bounded from below (since $\sgepsi$
is bi-directional).
Let $\eps_0\in\XX$.
The hypotheses of Lemma~\ref{lemcontbf} hold. Hence, $L$ is a contraction and \eqref{epsnonincbf} yields
\begin{align}
    \psX{L\eps_0}{\eps_0}
    = \norm{\sgepsi(T, 0)\eps_0}^2
    \leq \norm{\eps_0}^2
    - 2(r-\mu)\int_0^\tau\normY{C\sgepsi(t, 0)\eps_0}^2\dd t.
\end{align}
From there, the proof is identical to the proof of Theorem~\ref{thmainstr}, from equation \eqref{psL} to the end, by replacing $\tau$ by $T$, $\sgeps$ by $\sgepsi$ and $\Obsspace_\tau$ by $\XX$.
Hence, $L^n\eps_0\cv0$ as $n\cv\infty$, which ends the proof of Theorem~\ref{thmainbfstr}.
\end{proof}

\section{Examples and applications}\label{secex}

We provide two examples of applications of the main results of Section~\ref{secmain}. First, we consider the theoretical example of the one-dimensional time-varying transport equation with periodic boundary conditions. Then,
we apply the obtained results to a model of a batch crystallization process in order to reconstruct the Crystal Size Distribution (CSD) from the Chord Length Distribution (CLD).

\subsection{One-dimensional time-varying transport equation with periodic boundary conditions}
\label{sectrans}

As an example of the theory exposed in the former two sections we consider a one-dimensional time-varying transport equation with periodic boundary conditions. 
More precisely, let $\xf>\xd\geq0$ and $\XX = L^2((\xd,\xf);\R)$ the set of real-valued square-integrable functions over $(\xd, \xf)$, endowed with the inner product
$\psX{f}{g} = \int_{\xd}^{\xf}fg$ for all $f, g\in\XX$.

Let $\doma = \setst{\ff\in \XX}{\ff(\xd) = \ff(\xf), \ff'\in \XX}$
and $\vit\in C^1([0, T], \R)$.
For all $t\geq0$, let
\fonction{\opa(t)}{\doma}{\XX}{\ff}{-\vit(t)\displaystyle\frac{\dd \ff}{\dd x}.}
Then $\opa(t)$ is a skew-adjoint operator for all $t\geq0$.
Hence $(\opa(t))_{t\geq0}$ is a stable family of generators of strongly continuous groups that share the same domain $\doma$.
Moreover $t\mapsto A(t)f$ is continuously differentiable for all $f\in\doma$ since $\vit$ is of class $C^1$.
Then \cite[Chapter 5, Theorem 4.8]{Pazy} ensures that $(\opa(t))_{t\in[0, T]}$ is the generator of a unique bi-directional unitary (\ie forward and backward contraction) evolution system on $\XX$ denoted by $(\evol(t, s))_{0\leq s\leq t}$.
Moreover, $\sg(t, s)$ is defined for all $t\geq s\geq0$ and all $\etat_0\in\XX$ by
\begin{align}\label{eq_OpATransp}
    (\sg(t, s) \etat_0) (x) = \etat_0 (\trans(x, t, s)),
\end{align}
where
\begin{equation}\label{eq_Trans}
\trans(x, t, s) = \xd + \left(\left(x-\xd-\int_s^t\vit(\tau)\dd \tau\right) \modulo (\xf-\xd)\right)
\end{equation}
for almost all $x\in(\xd, \xf)$.

Hence, for all real Hilbert space $\YY$ and all output operator $C\in\linxy$,
the pair $((\opa(t))_{t\in[0, T]}, C)$ is $0$-weakly detectable, as well as
the pair $((-\opa(t))_{t\in[0, T]}, C)$.
Consequently, the transport equation with periodic boundary conditions is a good candidate to apply the observer methodology previously developed, in both the asymptotic or back and forth context.
Moreover, in the asymptotic context, we have the following proposition, which is useful to apply Theorem~\ref{thmain}.
\begin{proposition}\label{propito}
Assume that $T=+\infty$ and that both $\vit$ and its derivative $\vit'$ are bounded.
If there exist $\vit_\infty\in C^1(\R_+, \R)$ and an increasing positive sequence $(t_n)_{n\geq0}\cv+\infty$ such that $\vit(t_n+t)\cv\vit_\infty(t)$ as $n\cv+\infty$ for all $t\geq0$,
then
$\|\sg(t_n+t, t_n)-\sg_\infty(t, 0)\|_{\linx} \cv0$ as $n\cv+\infty$
uniformly in $t\in[0, \tau]$
for all $\tau\geq0$,
where $\sg_\infty$ is the evolution system generated by $\left(-\vit_\infty(t)\frac{\diff}{\diff x}\right)_{t\geq0}$.
\end{proposition}
In particular, note that if $\vit$ is periodic, then $\vit$ and $\vit'$ are bounded and there exits a \emph{bounded} sequence $(t_n)_{n\geq0}$ and a constant $\vit_\infty>0$ such that $\|\sg(t_n+t, t_n)-\sg_\infty(t)\|_{\linx} \cv0$ as $n\cv+\infty$
uniformly in $t\in[0, \tau]$
for all $\tau\geq0$,
where $\sg_\infty$ is the strongly continuous semigroup generated by $-\vit_\infty\frac{\diff}{\diff x}:\doma\to\XX$.
\begin{proof}[Proof of Proposition~\ref{propito}]
It is a direct application of \cite[Theorem 10.2.b]{ito}.
The consistency condition (C) of \cite{ito} is satisfied since for all $\etat_0\in\doma$,
\begin{align}
    A(t_n+t)\etat_0 = -\vit(t_n+t)\frac{\diff\etat_0}{\diff x}
    \cvl{n\cv+\infty} -\vit_\infty(t)\frac{\diff\etat_0}{\diff x}
\end{align}
Moreover,
$(\norm{A(t_n+t)\etat_0})_{n\geq0}$ 
is
bounded by
$\sup_{\R_+}\abs{G}\norm{\frac{\diff \etat_0}{\diff x}}$
for all $t\geq0$ and all $\etat_0\in\doma$.
For all 
$\etat_1, \etat_2\in\doma$, all $n\in\N$ and all $t, \tau\geq0$, we have the following inequalities:
\begin{align}
    \vert\langle A(t_n+t+\tau)\etat_1 - A(t_n+t)\etat_2&,\ \etat_1-\etat_2\rangle_\XX\vert\nonumber\\
    &\leq \abs{\psX{(A(t_n+t+\tau) - A(t_n+t))\etat_1}{\etat_1-\etat_2}}
    \nonumber\\
    &\quad + \abs{\psX{A(t_n+t)(\etat_1 -\etat_2)}{\etat_1-\etat_2}}\nonumber\\
    &\leq \abs{G(t_n+t+\tau)-G(t_n+t)}\norm{\frac{\diff \etat_1}{\diff x}}\norm{\etat_1-\etat_2}\tag{
    since $A(t_n+t)$ is skew-adjoint
    }\\
    &\leq \sup_{\R_+} \abs{G'} \tau \norm{\frac{\diff \etat_1}{\diff x}}\norm{\etat_1-\etat_2}.
\end{align}
Hence, the condition (E2u) of \cite{ito} is also satisfied.
Therefore, all the hypotheses of \cite[Theorem 10.2.b]{ito} are met, which ends the proof.
\end{proof}

In the following sections, the form of the output operator is investigated. The two considered forms will be of use in the application of the results to a crystallization process.

\subsubsection{Geometric conditions on the output operator}
\label{secgeom}
If the kernel of the output operator $\opc\in\linxy$ satisfies some geometric conditions, then the kernel of the observability Gramian of the system may be linked to the kernel of $\opc$. Indeed, assume that there exists a set $\interv \subset [\xd, \xf]$ \st
\begin{equation}\label{Ass_KerC}
\ker \opc = \setst{\ff\in X}{\ff\vert_{\interv} = 0},
\end{equation}
where $f\vert_U$ denotes the restriction of $f$ to $U$.
Then $\etat_0\in\ker \gram(t_0, \tau)$ for some $t_0, \tau \geq0$ if and only if
$
    \left(\sg(s, t_0)\etat_0\right)\vert_\interv = 0
$
for almost all $s\in(t_0, t_0+\tau)$,
\ie
$
    \etat_0(v(x, s, t_0)) = 0
$
for almost all $s\in(t_0, t_0+\tau)$ and almost all $x\in\interv$.
Hence
\begin{equation}\label{eq_WUmax}
\ker \gram(t_0, \tau) = \setst{\ff\in X}{\ff\vert_{\intervmax} = 0}
\end{equation}
where
$\intervmax = \set{\trans(x, s, t_0), x\in\interv, s\in[t_0, t_0+\tau]}$.
Moreover, note that
\begin{equation}
\ker \gram(t_0, \tau)^\perp = \setst{\ff\in X}{\ff\vert_{[x_0, x_1]\setminus\intervmax} = 0}.
\end{equation}
This leads to the following result. Roughly speaking, it states that if the observation time $\tau$ is sufficiently large for all the data to pass through the observation window $[\xmin, \xmax]$, then the observable part of the state is actually the full state.

\begin{proposition}\label{propgeom}
Assume that
$
\ker \opc
\subset \setst{\ff\in X}{\ff\vert_{[\xmin, \xmax]} = 0}
$
for some $[\xmin, \xmax]\subset[\xd, \xf]$.
If 
\begin{equation}
\abs{\int_{t_0}^{t_0+\tau}\vit(t)\dd t} \geq (\xf-\xd) - (\xmax-\xmin),    
\end{equation}
for some $t_0, \tau \geq0$,
then $\ker \gram(t_0, \tau) = \{0\}$.
\end{proposition}
\begin{proof}
According to \eqref{eq_WUmax}, it is sufficient to prove that $\intervmax = [\xd, \xf]$ when $\interv = [\xmin, \xmax]$.
Clearly, $\interv\subset\intervmax$.
Now, let $x\in[x_0, x_1]\setminus\interv$.
If $\int_{t_0}^{t_0+\tau}\vit(t)\dd t\geq0$, 
set $d = (\xmin - x) \modulo (\xf-\xd)$.
Then $d\leq (\xf-\xd) - (\xmax-\xmin)$.
Hence, according to the intermediate value theorem, there exists $s\in[t_0, t_0+\tau]$ such that $\int_{t_0}^{s}\vit(t)\dd t = d$.
Using \eqref{eq_Trans}, we obtain
$x = \trans(\xmin, s, t_0)$.
Otherwise,
$\int_{t_0}^{t_0+\tau}\vit(t)\dd t\leq0$.
Set $d = (x-\xmax) \modulo (\xf-\xd)$.
Similarly, $d\leq (\xf-\xd) - (\xmax-\xmin)$.
Hence, according to the intermediate value theorem, there exists $s\in[t_0, t_0+\tau]$ such that $\int_{t_0}^{s}\vit(t)\dd t = -d$.
Using \eqref{eq_Trans}, we obtain
$x = \trans(\xmax, s, t_0)$.
Thus, in both cases, there exists $\tilde{x}\in\interv$ and $s\in[t_0, t_0+\tau]$ such that
$x = \trans(\tilde{x}, s, t_0)$.
\end{proof}

\subsubsection{Integral output operator with bounded kernel}
\label{secintout}

Assume that the output operator $\opc\in\linxy$ is an integral output operator with bounded kernel, that is, there exists $k\in L^\infty((\xd, \xf); Y)$
(\ie with $\esssup_{x\in(\xd, \xf)} \normY{k(x)} < +\infty)$
such that\footnote{
$C$ is well-defined because $[x_0, x_1]\ni x\mapsto k(x)f(x)$ is Bochner integrable, since $x\mapsto\|k(x)\|_Y$ is bounded and $x\mapsto f(x)$ is integrable (since $(x_0, x_1)$ has finite length and $f\in L^2((x_0, x_1); \R)$).
}
\begin{equation}\label{eq_OutOp}
\opc \ff = \int_{\xd}^{\xf} \noy(x) \ff(x) \dd x
\end{equation}
for all $\ff\in\XX$.
Then, there is no time interval $(t_0, t_0+\tau)\subset\R_+$ such that the pair $((A(t))_{t\geq0}, C)$ is exactly observable on $(t_0, t_0+\tau)$.
\begin{proposition}\label{propker}
If $\opc\in\linxy$ satisfies \eqref{eq_OutOp} for some $k\in L^\infty((\xd, \xf); Y)$, then
for all $t_0, \tau\geq0$ and all $\delta>0$, there exists $\etat_0\in\XX$ such that
\begin{align}
    \psX{\gram(t_0, \tau)\etat_0}{\etat_0} \leq \delta \norm{z_0}^2.
\end{align}
\end{proposition}
Hence, for such output operators, the convergence of an observer must rely on weaker observability assumptions, such as the approximate observability.
In the application of the results to a crystallization process (see Section~\ref{seccryst}), the reader will find that $\opc$ is precisely an integral output operator with bounded kernel. This justifies the whole approach of the paper, since our results are based on such weaker observability hypotheses (namely approximate observability and not exact observability).

\begin{proof}[Proof of Proposition~\ref{propker}]
Let $t_0, \tau\geq0$, $\etat_0\in\XX$ and $\etat(t) = \sg(t_0+t, t_0) z_0$ for all $t\geq t_0$. 
Since $(\xd, \xf)$ is bounded,
any $\ff\in L^2((\xd, \xf); \R)$ is also integrable. Set $\|\ff\|_{L^1((\xd, \xf); \R)}
=\int_{\xd}^{\xf}\abs{\ff(x)}dx
$.
Then
\begin{align}
\psX{\gram(t_0, \tau)\etat_0}{\etat_0}
  &=\int_{t_0}^{t_0+\tau} \normY{C\etat(t)}^2\dd t \nonumber\\
  &\leq\int_{t_0}^{t_0+\tau} \left(\int_{\xd}^{\xf}\normY{\noy(x)\etat(t, x)}\dd x\right)^2\dd t
  \tag{
  Bochner inequality}
  \\
  &\leq \int_{t_0}^{t_0+\tau} \left(\int_{\xd}^{\xf}\normY{\noy(x)}\abs{\etat(t, x)}\dd x\right)^2\dd t
  \nonumber
  \\
  &\leq \|\noy\|^2_{L^\infty((\xd, \xf); Y)} \int_{t_0}^{t_0+\tau} \left(\int_{\xd}^{\xf}\abs{\etat(t, x)}dx\right)^2\dd t \nonumber\\
  &\leq \tau \|\noy\|^2_{L^\infty((\xd, \xf); Y)}
  \sup_{t\in[t_0, t_0+\tau]}
  \|\etat(t)\|^2_{L^1((\xd, \xf); \R)}. \nonumber
\end{align}
Moreover, by the usual transport properties of $v$, we get for all $t\in[t_0, t_0+\tau]$ that
\begin{align*}
    \|\etat(t)\|^2_{L^1((\xd, \xf); \R)}
    = \|\etat_0(v(t, t_0, \cdot))\|^2_{L^1((\xd, \xf); \R)}
    = \|\etat_0\|^2_{L^1((\xd, \xf); \R)}.
\end{align*}
Hence
\begin{align*}
\psX{\gram(t_0, \tau)\etat_0}{\etat_0}
  \leq \tau \|\noy\|_{L^\infty((\xd, \xf); Y)} \|\etat_0\|^2_{L^1((\xd, \xf); \R)}.
\end{align*}
The result follows from the fact that the norms $\|\cdot\|_{L^1((\xd, \xf); \R)}$ and $\|\cdot\|_{L^2((\xd, \xf); \R)}$ are not equivalent.
\end{proof}

\begin{remark}
According to Remark~\ref{rmcca}, the boundedness of $C^*CA$ as an operator from $(\doma, \norm{\cdot})$ to $(\XX, \norm{\cdot})$ is an interesting property for the convergence to $0$ of the correction term $C\eps$ of the observers.
If we ask more regularity to the solutions of the transport equation, then the integral output operators in the form of \eqref{eq_OutOp} satisfy this assumption. Indeed, assume (in this remark \emph{only}) that $\XX =\{\ff\in L^2(\xd, \xf; \R)\mid f'\in L^2(\xd, \xf; \R)\}$ endowed with the inner product $\psX{f}{g}=\int_{\xd}^{\xf}(fg+f'g')$ and $\doma_\text{new} = \{\ff\in\XX\mid \ff(\xf) = \ff(\xf), \ff'(\xf) = \ff'(\xf), f''\in L^2(\xd, \xf; \R)\}$.
Then, for all $\etat_0\in\doma_\text{new}$,
\begin{align*}
    \normY{CA\etat_0}^2
    &\leq \left(\int_{\xd}^{\xf} \normY{k(x)\frac{\diff \etat_0}{\diff x}(x)}dx\right)^2\\
    &\leq \|k\|_{L^\infty((\xd,\xf),\YY)}  \left(\int_{\xd}^{\xf} \left|\frac{\diff \etat_0}{\diff x}(x)\right|dx\right)^2\\
    &\leq \|k\|_{L^\infty((\xd,\xf),\YY)} (\xf-\xd) \norm{\etat_0}^2
\end{align*}
by the Cauchy-Schwarz inequality.
Thus, $C^*CA\in\lin((\doma_\text{new}, \norm{\cdot}), (\XX, \norm{\cdot}))$ since $\opc$ is bounded.
\end{remark}

\subsection{Estimation of the CSD from the CLD in a batch crystallization process}\label{seccryst}

\subsubsection{Modeling the batch crystallization process}

In the chemical and pharmaceutical industries, the crystallization process is one of the simplest and cheapest way to produce some pure solid. In order to control the physical and chemical properties of the product, the control of the Crystal Size Distribution (CSD) is of major importance.
Since there is no effective measurement method able to determine the CSD online during the process, the estimation of the CSD based on other measurements is a crucial issue.
We consider a batch crystallization process.
The reader may refer to \cite{Mullin, Mersmann} for a theoretical analysis of this process.
One of the simplest model of the crystal growth during the process is the following population balance equation:
\begin{equation}
\begin{aligned}
\begin{cases}
\displaystyle\frac{\partial \dtc}{\partial t}(t, x) + \vit(t) \frac{\partial \dtc}{\partial x}(t, x) = 0,
\quad \forall (t, x) \in [0, T]\times[\xmin, \xmax]\\
\dtc(0, \cdot) = \dtc_0\\
\dtc(\cdot, \xmin) = \cont,
\end{cases}
\end{aligned}
\label{systbilan}
\end{equation}
with the following notations:
\begin{itemize}
    \item $T$ is the experiment duration;
    \item $[\xmin, \xmax]$ is the crystal size range: all crystals are assumed to be spherical with radius $x\in[\xmin, \xmax]$ where $\xmax>\xmin>0$;
    \item $x\mapsto\dtc(t, x)$ is the CSD at time $t$;
    \item $\vit$ is the growth kinetic, assumed size independent (McCabe hypothesis);
    \item $\cont$ represents the nucleation. All new crystals have size $\xmin$.
\end{itemize}
Here $\vit$ is supposed to be known, contrary to $\cont$ and $\dtc$.
In practice, $G$ can be estimated via a simple model 
based on the solute concentration and the solubility thanks to solute concentration and temperature sensors (see, \eg \cite{brivadisjpc, Uccheddu}, or \cite{Mersmann, Mullin} for more detailed models).
We reformulate~\eqref{systbilan} in order to match our theoretical results.
The size of the crystals is supposed to be increasing, \ie $\vit(t)>0$ for all $t\in[0, T]$.
Assume that the maximal crystal size $\xmax$ is never reached by any crystals in time $T$, \ie $\dtc(t, \xmax) = 0$ for all $t\in[0, T]$.

Let $\xd = \xmin - \int_0^T G(s)\dd s$ and $\xf = \xmax$.
We introduce the initial state variable $\etat_0$, given for all $x\in[\xd, \xf]$ by
\begin{equation}\label{eq_etat0trans}
\etat_0(x) = \begin{cases}
\cont\left(
\frac{T(\xmin-x)}{\int_0^T\vit(s)\dd s}
\right) & \text{ if } \xd\leq x \leq \xmin,\\
\dtc_0(x) & \text{otherwise}.
\end{cases}
\end{equation}
Let $\XX = L^2((\xd, \xf);\R)$.
As in Section~\ref{sectrans},
there exists a unique solution $\etat\in C^0([0, T]; \XX)$
to the abstract Cauchy problem
\begin{equation}
\begin{aligned}
\begin{cases}
\displaystyle\dot{\etat}(t, x) =  -\vit(t) \frac{\partial \etat}{\partial x}(t, x) & \forall (t, x) \in [0, T]\times[\xd, \xf],\\
\etat(0) = \etat_0\\
\end{cases}
\end{aligned}
\label{systapp}
\end{equation}
Moreover, \eqref{eq_OpATransp} and \eqref{eq_Trans} combined with \eqref{eq_etat0trans} yield
$$\etat(t, \xmin)
=
\etat_0(\xmin)
= u(t)$$ for all $t\in[0, T]$.
Hence,
$    \etat(t, x)
    =
    \dtc(t, x)
$
for all $t\in[0, T]$ and all $x\in[\xmin, \xmax]$.

We are now in the context developed in the previous section of the one-dimensional transport equation with periodic boundary conditions (since the right boundary term does not influence $\etat(t, \xmin)$ on the time interval $[0, T]$).
Our goal is to reconstruct offline the initial CSD $\dtc_0 = \etat_0\vert_{[\xmin, \xmax]}$ thanks to the BFN algorithm.
We now introduce an output operator $\opc$.

\subsubsection{Modeling the \texorpdfstring{\fbrm}~technology}

The focused beam reflectance measurement (\fbrm) technology is an \emph{in situ} sensor that measures data online during a crystallization process.
The probe is equipped with a laser beam in rotation that scans across the particles.
While the beam hit a particle, light is backscattered to the probe.
The sensor counts the number of distinct light pulses and their duration. For each pulse, a length on a particle (\ie a chord length) can be determined, since the rotation speed of the beam is known and the speed of the particle is supposed to be insignificant. Hence, one can deduce the Chord Length Distribution (CLD) of the particles. The reader may refer to \cite{LW, Simmons, barrett1999line} for more details about this technology, and how it is linked to the CLD.

Using the \fbrm technology to recover the CSD is a major current issue in process engineering. To the best of the authors' knowledge, all the proposed estimation techniques (see, \eg  \cite{Agi, LW, Simmons}) are based on inverse problem methods, such as the Tikhonov regularization, in order to inverse the function that maps the CSD to the CLD.
These approaches are not based on any crystal growth model, and ignore the dynamics of the system.
On the contrary, in this paper, both the \fbrm technology and a crystal growth model are used to reconstruct the CSD, by using an observer of the dynamical system \eqref{systbilan} and using the CLD as an output of the system.

At a fixed time $t \in[0, T]$, for a given CSD $\dtc(t, \cdot)$ of spherical particles, the corresponding cumulative CLD $\dlc(t, \cdot)$ supposed to be measured by the \fbrm probe can be written as
\begin{align}
    \dlc(t, \ell) = \int_{\xmin}^{\xmax}\noy(x, \ell)\dtc(t, x) \dd x, \quad \forall \ell \in [0, 2\xmax],
\end{align}
where $\ell$ represents the length of a chord and $\noy$, defined in \cite{moi, LW}, satisfies
\begin{align}
    \noy(x, \ell) = 1 - \I_{[0, 2x[}(\ell)\sqrt{1-\left(\frac{\ell}{2x}\right)^2}, \quad \forall (\ell, x) \in [0, 2\xmax]\times [\xmin, \xmax],
\end{align}
where $\I_{[0, 2x[}$ is the characteristic function of $[0, 2x)$.
Set $\YY = L^2((\lmin, \lmax); \R)$ with $\lmin = 0$ and $\lmax = 2\xmax$. Let $\opc\in \linxy$ be defined by
\fonction{\opc}{\XX}{\YY}{\ff}{\ell\mapsto\ps{\noy(\cdot, \ell)}{\ff\vert_{[\xmin, \xmax]}}_{L^2((\xmin, \xmax); \R)}.}
Then $\noy\in L^\infty((\xmin, \xmax); \YY)$ since $0\leq \noy(x, \ell)\leq 1$
for all $(x, \ell) \in [\xmin, \xmax]\times[0, 2\xmax]$.
Thus, $\opc$ is a well-defined integral operator with kernel $\noy$ and, according to Section~\ref{secintout}, there is no time interval $(t_0, t_0+\tau)\subset[0, T]$ on which the system is exactly observable.
It remains to analyse $\ker\opc$.
\begin{proposition}\label{propC}
The kernel of the integral operator $\opc$ is given by
\begin{equation}
\ker \opc = \setst{\ff\in X}{\ff\vert_{[\xmin, \xmax]} = 0}.
\end{equation}
\end{proposition}
Therefore, one can apply the results of Section~\ref{secgeom}, and in particular Proposition~\ref{propgeom}, to the pair $((A(t))_{t\in[0, T]}, \opc)$.
According to the definition of $\xd$ and $\xf$, $\int_0^T G(t)\dd t = (\xf-\xd) - (\xmin-\xmax)$.
Hence, $\gram(0, T)$ is injective.
Thus, according to Theorem~\ref{thmainbf}, the system \eqref{obs2}-\eqref{obs2b} is a weak back and forth observer of~\eqref{syst}. Moreover, since $A(t)$ is skew-adjoint for all $t\in[0, T]$, Theorem~\ref{thmainbfstr}~\ref{notskew} also applies.
Hence, the BFN algorithm reconstructs the CSD from the CLD in the strong topology.

\begin{proof}[Proof of Proposition~\ref{propC}]
Clearly, $\ker \ceps \supset \setst{\ff\in\XX}{\ff\vert_{[\xmin, \xmax]} = 0}$.
Let $\ff\in\ker\ceps$. We want to show that $\ff\vert_{[\xmin, \xmax]} = 0$. For almost all $\ell\in(0, 2\xmin)$ we have
\begin{align}
0 &=\int_{\xmin}^{\xmax} \noy(\ell, x)\ff(x)\dd x\nonumber \\
&= \int_{\xmin}^{\xmax} \ff(x)\dd x - \int_{\xmin}^{\xmax} \ff(x)\sqrt{1-\left(\frac{\ell}{2x}\right)^2}\dd x.\label{eqnulK}
\end{align}
In order to apply the Leibniz integral rule on $(0, 2\xmin)$, we check that
\begin{itemize}
\item for all $\ell\in (0, 2\xmin),\ x\mapsto\ff(x)\sqrt{1-\left(\frac{\ell}{2x}\right)^2}$ is integrable on $(\xmin, \xmax)$,
\item for all $x\in (\xmin, \xmax),\ \ell\mapsto\ff(x)\sqrt{1-\left(\frac{\ell}{2x}\right)^2}$ is $C^\infty$ on $(0, 2\xmin)$.
\end{itemize}
Hence, $\ceps \ff$ is $C^\infty$ on $(0, 2\xmin)$. Since $\ceps\ff = 0$ almost everywhere on $(0, 2\xmin)$, we get that
\begin{align*}
    (\ceps \ff) ^{(n)}(0) = 0, \qquad \forall n\in \N.
\end{align*}
In the following, we determine an expression of $(\ceps \ff) ^{(2n)}(0)$.
Fix $x\in(\xmin, \xmax)$.
Set \fonction{\f}{(0, 2\xmin)}{\R}{\ell}{-\sqrt{4x^2-\ell^2}.}
We show by induction that for all $n \geq 1$, there exists a family $(a^n_{i, j})_{i, j \in \N} \in (\R_+)^{(\N^2)}$ such that:
\begin{itemize}
\item the set $\{(i, j)\in \N^2\ \mid\ a^n_{i, j} \neq 0\}$ is finite,
\item $a^n_{0, n-1} \neq 0$,
\item $\forall j \in \N\backslash\{n-1\},\ a^n_{0, j} = 0$,
\item $\f^{(2n)}(\ell) = \sum\limits_{i, j \in \N} a^n_{i, j}\ell^i(4x^2-\ell^2)^{-\frac{2j+1}{2}}$ for all $\ell\in(0, 2\xmin)$.
\end{itemize}
\paragraph{Base case}For all $\ell\in (0, 2\xmin)$,
\begin{align*}
\f'(\ell) &= \ell (4x^2-\ell^2)^{-\frac{1}{2}},\\
\f^{(2)}(\ell) &= (4x^2-\ell^2)^{-\frac{1}{2}} + \ell^2 (4x^2-\ell^2)^{-\frac{3}{2}}.
\end{align*}
Then, it is sufficient to set, for all $(i, j)\in (\N^*)^2$,
\begin{align*}
a^1_{i, j} =
\begin{cases}
1 & \text{if } (i, j) \in \{(0, 1), (2, 2)\}\\
0 & \text{else}
\end{cases}
\end{align*}
\paragraph{Inductive step} Let $n \geq 1$. Assume there exists such a family $(a^n_{i, j})_{i, j \in \N}$. We need to compute $\f^{(2(n+1))}$. For all $\ell\in (0, 2\xmin)$,
\begin{align*}
\f^{(2n)}(\ell) &= a_{0, n-1}(4x^2-\ell^2)^{-\frac{2(n-1)+1}{2}} + \sum\limits_{i \geq1, j \geq0} a^n_{i, j}\ell^i(4x^2-\ell^2)^{-\frac{2j+1}{2}} \ \text{(by hypothesis).}
\end{align*}
Computing the next two derivatives of $\f^{(2n)}$, we get
\begin{align*}
\f^{(2n+1)}(\ell)
&= (2(n-1)+1)a_{0, n-1}\ell(4x^2-\ell^2)^{-\frac{2n+1}{2}}\\
&\quad + \sum\limits_{i\geq1, j \geq0} (2j+1)a^n_{i, j}\ell^{i+1}(4x^2-\ell^2)^{-\frac{2(j+1)+1}{2}}\\
&\qquad + \sum\limits_{j \geq0} a^n_{1, j}(4x^2-\ell^2)^{-\frac{2j+1}{2}} + \sum\limits_{i\geq2, j \geq0} i a^n_{i, j}\ell^{i-1}(4x^2-\ell^2)^{-\frac{2j+1}{2}}
\end{align*}
and
\begin{align*}
\f^{(2n+2)}(\ell)
&= (2(n-1)+1)a_{0, n-1}(4x^2-\ell^2)^{-\frac{2n+1}{2}}\\
&\quad + \sum\limits_{j \geq1} (2j-1)a^n_{1, j-1}\ell(4x^2-\ell^2)^{-\frac{2j+1}{2}}\\
&\qquad + \sum\limits_{i\geq3, j \geq2} (2(j-1)+1)(2j-3)a^n_{i-2, j-2}\ell^i(4x^2-\ell^2)^{-\frac{2j+1}{2}}\\
&\qquad\quad + \sum\limits_{i\geq1, j \geq1} (i+1) (2j-1)a^n_{i, j-1}\ell^i(4x^2-\ell^2)^{-\frac{2j+1}{2}}\\
&\quad + (2(n-1)+1)(2n+1)a_{0, n-1}\ell^2(4x^2-\ell^2)^{-\frac{2(n+1)+1}{2}}\\
&\qquad + \sum\limits_{i\geq0, j \geq0} (i+1)(i+2) a^n_{i+2, j}\ell^i(4x^2-\ell^2)^{-\frac{2j+1}{2}}\\
&\qquad\quad + \sum\limits_{i\geq2, j \geq1} (2j-1)i a^n_{i, j-1}\ell^i(4x^2-\ell^2)^{-\frac{2j+1}{2}}.
\end{align*}
For all $(i, j)\in \N^2$, set
\begin{align*}
a^{n+1}_{i, j} &=
(2n-1)a_{0, n-1} \I_{\{(0, n)\}}(i, j)\\
&\quad + (2n-1)(2n+1)a_{0, n-1} \I_{\{1\}\times\lc1, +\infty)}(i, j)\\
&\qquad+ (2j-1)(2j-3)a^n_{i-2, j-2} \I_{\lc3, +\infty) \times\lc2, +\infty)}(i, j)\\
&\qquad\quad+ (i+1) (2j-1)a^n_{i, j-1} \I_{\lc1, +\infty) \times\lc1, +\infty)}(i, j)\\
&\quad + (2(n-1)+1)(2n+1)a_{0, n-1} \I_{\{(2, n+1)\}}(i, j)\\
&\qquad+ (2j-1)i a^n_{i, j-1} \I_{\lc2, +\infty) \times\lc1, +\infty)}(i, j)\\
&\qquad\quad+ (i+1)(i+2) a^n_{i+2, j}.
\end{align*}
Then, to conclude the induction, one can check that
\begin{itemize}
\item for all $\{(i, j)\in \N^2\, a^{n+1}_{i, j} \geq 0$ since $(a^n_{i, j})_{i, j \in \N} \in (\R_+)^{(\N^2)}$,
\item $\{(i, j)\in \N^2 \mid a^{n+1}_{i, j} \neq 0\}$ is finite since $\{(i, j)\in \N^2 \mid a^n_{i, j} \neq 0\}$ is finite,
\item $a^{n+1}_{0, n} \geq (2n-1)a^n_{0, n-1} > 0$,
\item $\forall j \in \N\backslash\{n-1\},\ a^{n+1}_{0, j} = 0$,
\item $\f^{(2(n+1))}(\ell) = \sum\limits_{i, j \in \N} a^{n+1}_{i, j}\ell^i(4x^2-\ell^2)^{-\frac{2j+1}{2}}$ for all $\ell\in(0, 2\xmin)$.
\end{itemize}
Thus, since $(\ceps\ff)^{(2n)}(0) = 0$ for all $n\in \N^*$,
\begin{align}
0 &= \int_{\xmin}^{\xmax} a^n_{0, n-1}\frac{\ff(x)}{(2x)^{2n}}\dd x
\end{align}
for some $a^n_{0, n-1} > 0$. Let $n\in\N^*$. Then,
\begin{align*}
0 &= \int_{\xmin}^{\xmax} \frac{\ff(x)}{x^{2n}}\dd x \nonumber\\
&= \int_{\frac{1}{\xmax}}^{\frac{1}{\xmin}} \ff\left(\frac{1}{\til{x}}\right)\til{x}^{2n-2}\dd \til{x}. \qquad (\til{x} = \frac{1}{x})
\end{align*}
Set $\til{\ff}: [\frac{1}{\xmax}, \frac{1}{\xmin}] \ni \til{x} \longmapsto \ff(\frac{1}{\til{x}})$. Then,
\begin{align}
0 &=\int_{\frac{1}{\xmax}}^{\frac{1}{\xmin}} \til{\ff}(\til{x})\til{x}^{2n-2}\dd \til{x} \nonumber\\
&= \frac{1}{2}\int_{\frac{1}{\xmax^2}}^{\frac{1}{\xmin^2}} \frac{\til{\ff}(\sqrt{\bar{x}})}{\sqrt{\bar{x}}}\bar{x}^{n-1}\dd \bar{x}. \qquad (\bar{x} = \til{x}^2)
\end{align}
Set $\bar{\ff}: [\frac{1}{\xmax^2}, \frac{1}{\xmin^2}] \ni \bar{x} \longmapsto \frac{\tilde{\ff}(\sqrt{\bar{x}})}{\sqrt{\bar{x}}}$.
Then we have
\begin{align}
0= \int_{\frac{1}{\xmax^2}}^{\frac{1}{\xmin^2}} \bar{\ff}(x)\bar{x}^{n-1}\dd \bar{x}.
\end{align}
Since the family $(x\mapsto x^n)_{n\geq0}$ is a total family in $L^2\left(\left(\frac{1}{\xmax^2}, \frac{1}{\xmin^2}\right); \R\right)$ from the Weierstrass approximation theorem, $\bar{\ff}=0$. Hence $\ff\vert_{(\xmin, \xmax)}=0$, which concludes the proof.
\end{proof}

\section*{Acknowledgments}
The authors would like to thank \'{E}lodie Chabanon, \'{E}milie Gagnière and Noureddine Lebaz for many fruitful discussions about crystallization processes and the \fbrm technology.

\bibliographystyle{plain}
\bibliography{references}
\end{document}